\documentclass{egpubl}
\usepackage{egsgp17}

\SpecialIssuePaper         
 \electronicVersion 

\ifpdf \usepackage[pdftex]{graphicx} \pdfcompresslevel=9
\else \usepackage[dvips]{graphicx} \fi

 \PrintedOrElectronic
 \usepackage{t1enc,dfadobe}

 \usepackage{egweblnk}
 \usepackage{cite}

 \author[K. Kar{\v c}iauskas \& J. Peters]
 {K{\c e}stutis Kar{\v c}iauskas$^a$ and J\"org Peters$^b$\\
 $^a$ Vilnius University\quad $^b$ University of Florida}

\usepackage{times}
\usepackage{amsmath}
\usepackage{amssymb}
\usepackage{amsfonts}
\usepackage{epsfig}
\usepackage{psfrag}
\usepackage[usenames,dvipsnames]{color}
\usepackage{subfigure}
\usepackage{overpic}
\usepackage{float}

\newcommand{\chg}[1]{\textcolor{Brown}{#1}}

\newcommand{\figref}[1]{Fig.~\ref{#1}}
\newcommand{\secref}[1]{Section~\ref{#1}}

\newcommand{\ga}{\alpha}   
\newcommand{\gb}{\beta}

\newcommand{\gl}{\lambda} 
\newcommand{\gs}{\sigma}

\newcommand{\charm}{\chi}  
\newcommand{\NN}{N}  

\newcommand{\bb}{\mathbf{b}}   
\newcommand{\bc}{\mathbf{c}}   
\newcommand{\bp}{\mathbf{p}}   

\newtheorem{prop}{Proposition}

\def\val{n}  

\newcommand{\IL}{\textit{left}}

\newcommand{\IR}{\textit{right}}
\newcommand{\IB}{\textit{bottom}}
\newcommand{\IT}{\textit{top}}

\def\bbr{\mathbb {R}}

\def\ac{\mathsf{c}}

\def\bpg{\grave{\mathbf{p}}}
\def\bpa{\acute{\mathbf{p}}}

\newcommand{\dg}{d}

\newcommand{\pc}[4]{\pmb{[}#1\pmb{]}^{#2}_{#3}#4}
\def\mmm{\mathbf{m}}

\def\ggg{\mathbf{g}}

\newcommand{\gfree}{{\mathcal{P}}}

\def\mmm{m}

\newcommand{\kbar}{\bar{k}}
\newcommand{\khat}{\hat{k}}
\newcommand{\nnp}{n^+}
\newcommand{\GGG}{\mathbf{G}}

\newcommand{\eigv}{w}
\newcommand{\sixcap}{\mathbf{\hat g}}

\newcommand{\bfg}{\grave{\mathbf{f}}}
\newcommand{\bfa}{\acute{\mathbf{f}}}
\newcommand{\xeq}{eq}
\newcommand{\dap}{d}
\newcommand{\gss}{guided subdivision surface}
\newcommand{\hl}{highlight line}
\newcommand{\hld}{\hl\ distribution}
\newcommand{\CCa}{Catmull-Clark}
\newcommand{\fn}{\bp}
\newcommand{\bz}[2]{\fn^{#1}_{#2}}

\newcommand{\eon}{irregular node}
\newcommand{\cnet}{c-net}
\newcommand{\KK}{\kappa}

\newcommand{\xxb}{\mathbf{x}}
\newcommand{\charmt}{\tilde\chi}
\newcommand{\llk}{l}
 
\begin{document}
\def\wid{0.2\linewidth}
\def\wids{0.15\linewidth}
\def\skp{\hskip 0.05\linewidth}
\teaser{
  \centering
   \subfigure[control net]{
   \includegraphics[width=\wid]{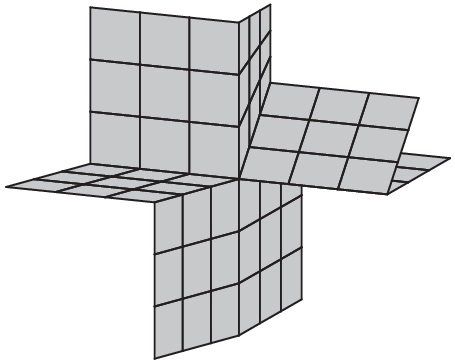}
   }
   \subfigure[guide]{
   \includegraphics[width=\wid]{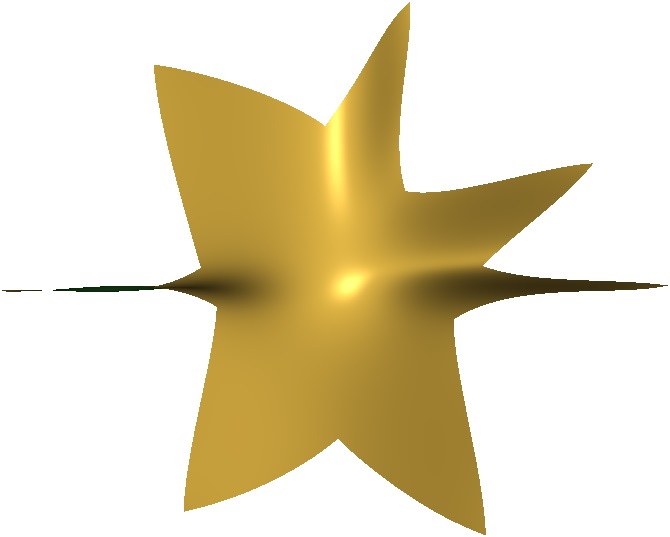}
   }
   \subfigure[subdivision rings]{
   \includegraphics[width=\wids]{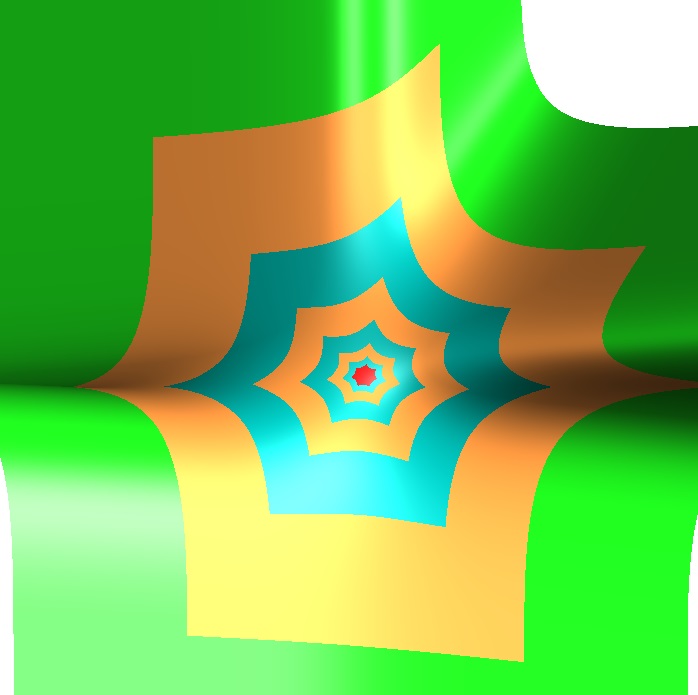}
   }
   \skp
   \subfigure[embossed detail]{
   \includegraphics[width=\wids]{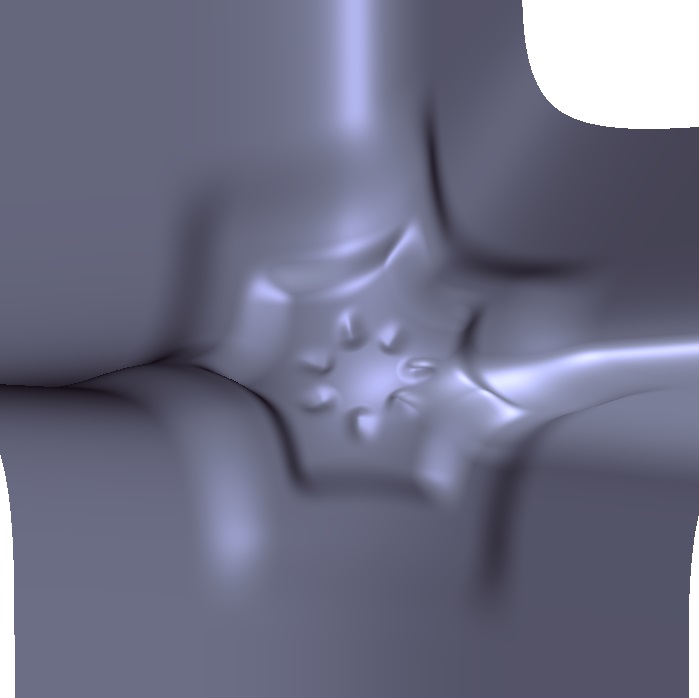}
   }
   \subfigure[\CCa]{
   \includegraphics[width=\wids]{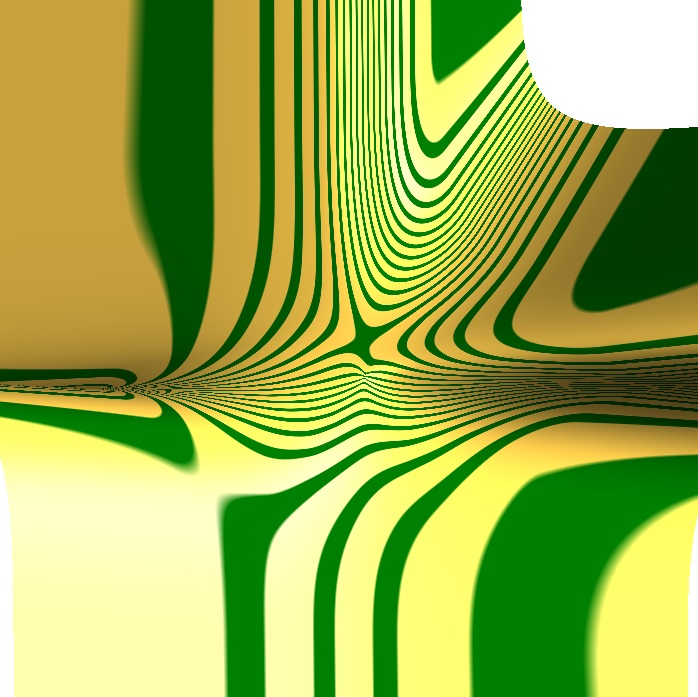}
   \skp
   \includegraphics[width=\wid]{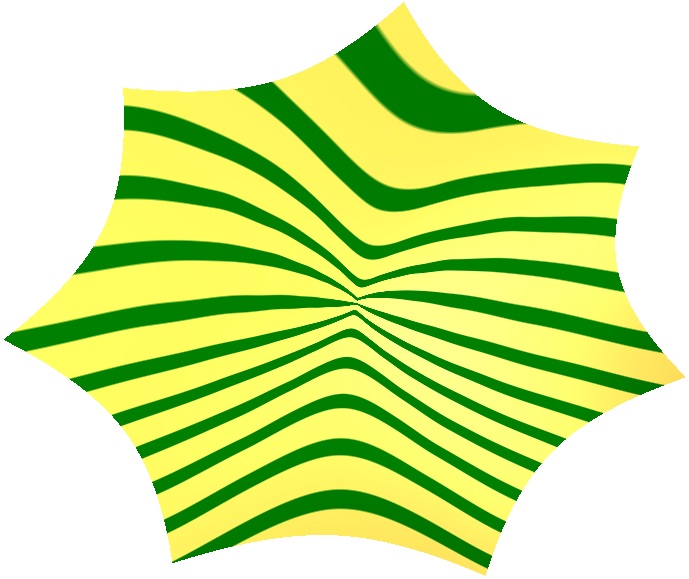}
   }
   \subfigure[guided subdivision]{
   \includegraphics[width=\wids]{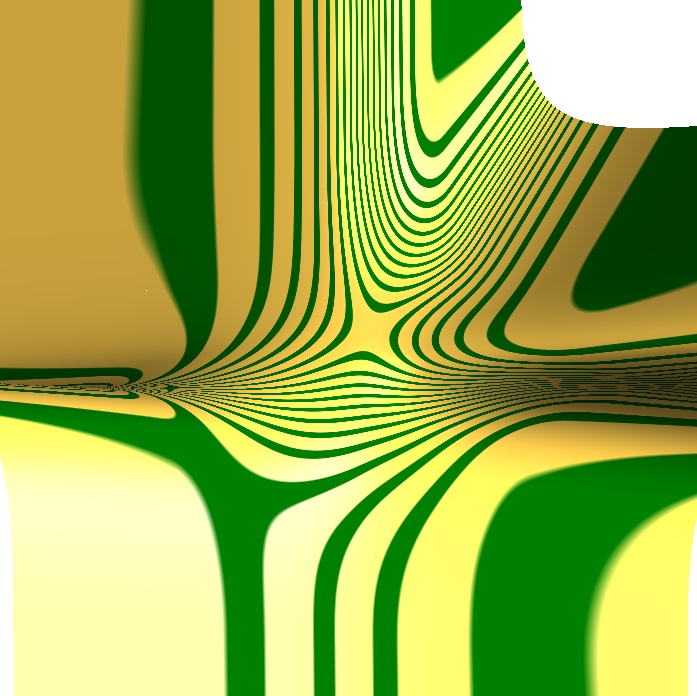}
   \skp
   \includegraphics[width=\wid]{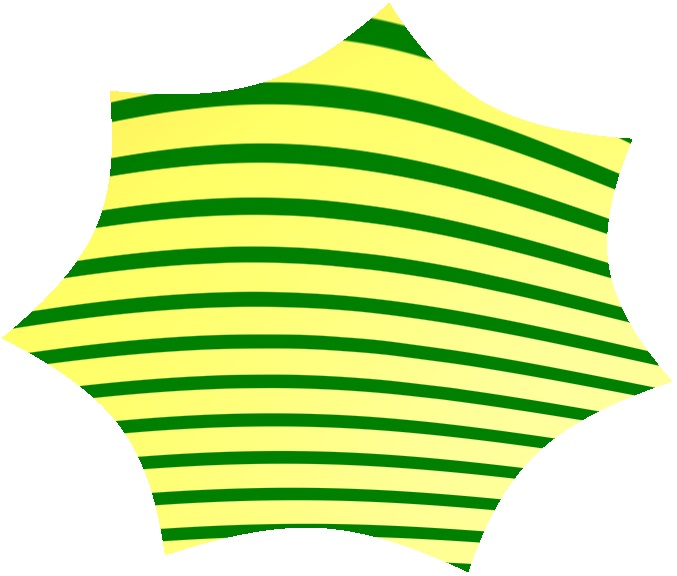}
   }
  \caption{Construction and comparison of guided subdivision: 
  (a) input net defining in (c) the irregular subdivision region and 
  a regular (green) bi-3 region; (b) guide
  surface that does not match the bi-3 spline but defines
  the shape; (c) six subdivision rings (alternating gold and cyan)
  capped by a finite polynomial (red) surface cap;
  (d) embossing exploiting the degrees of freedom in 
  the subdivision rings.
  (e) \CCa\ vs (f) guided subdivision: improvement of the \hld. 
   }
\label{fig:teaser}
}
\title[Guided Subdivision]%
{Guided subdivision surfaces: modeling, shape and refinability}

\maketitle

\begin{abstract}
Converting quad meshes to smooth manifolds,
guided subdivision offers a way to combine the good \hld s of
recent G-spline constructions with the refinability of 
subdivision surfaces.
Specifically, we present a $C^2$ subdivision algorithm of
polynomial degree bi-6 and a curvature bounded algorithm
of degree bi-5. 
We prove that the common eigen-structure of this class of
subdivision algorithms is determined by their guide and 
demonstrate that the eigenspectrum (speed of contraction) 
can be adjusted without harming the shape.
\begin{classification} 
\end{classification}
\end{abstract}

\section{Introduction}
The automatic conversion of regular quad meshes 
into $C^2$ surfaces of good shape and with built-in refinability
addresses challenges of design and engineering analysis.
Catmull-Clark subdivision offers refinability but the \hld s 
of the resulting surfaces are often deficient,
see \figref{fig:teaser}e. A decade-old technique, 
called guided subdivision \cite{Karciauskas:2006:CTM}, 
addresses in principle both shape and refinability when 
based on a high-quality guide surface. 
The underlying idea is to sample Hermite data from a 
well-shaped surface that need not match the smoothness or 
connectivity requirements of the final output surface.
However inheriting the shape of the guide is not straightforward.
A careful reparameterization is required since derivatives of the 
guide surface may not correspond to the natural
directions of the subdivision sequence of contracting
surface rings.

The proposed $C^2$ subdivision (and its visually identical
curvature-bounded counterpart of degree bi-5) leverage the guide shape of
a recent G-spline construction, \cite{Kestutis:2015:ISM}.
Highlighting the distinction between shape and mathematical
smoothness, the derived guide neither matches the 
surface surrounding the subdivision construction nor does
it generate $C^2$ surfaces; yet the derived \gss\ is 
$C^2$ up to the central point; and it is $C^2$ at the point
if the degree is bi-6, and the curvature is bounded at the 
center if the degree is bi-5, bi-4 or bi-3.

All new (bi-6, bi-5, bi-4 or bi-3) guided subdivision algorithms have a common
simple eigenstructure that is determined by the guide. 
In particular, the eigenspectrum and with it the speed of contraction
can be adjusted without harming the shape. For example, 
the surface rings that make up the subdivision surface 
can be made to uniformly contract by 1/2 regardless
of the valence.

\subsection{Related Literature}
Starting with \cite{doo78,Catmull-1978-CC} (and \cite{Loop:1987:SSS}
for triangulations), subdivision surfaces have become dominant in the 
animation industry (see e.g.\ \cite{DeRose:1998:SSI,journals/tog/NiessnerLMD12})
and have inspired many polyhedral geometry processing algorithms
and vice versa \cite{journals/cga/WarrenS04}.
Shape control, e.g.\ via semi-smooth creases, have had a stronger impact
than formal smoothness: except in pockets of academia, extensions 
of \CCa-subdivision to $C^2$ continuity
\cite{journals/tog/Levin06,Zorin:2006:CCC,Karciauskas:2006:CTM}
have largely been ignored.

Two new developments re-kindle the interest in higher-order smooth
subdivision: isogeometric computations on surfaces and advances in
quad meshing.
Following the early lead of \cite{Cirak:2000:SSP} more recently
subdivision surfaces have been used as computational domains,
see e.g.\
\cite{barendrecht2013isogeometric,journals/cma/PanXXZ16,%
journals/amc/Riffnaller-Schiefer16,juttler:subdiv:2016}.
Quad meshing \cite{QUADSTAR2012,VCDPBHB16}
has matured over the past decade starting with
\cite{Alliez:2003,Marinov:2004,journals/cgf/KalbererNP07}
leveraging
directional fields  \cite{Myles:2014,Pietroni:2016}.
Such meshes provide natural input for 
\gss s. The use of \gss\ can be traced back to 
\cite{journals/tog/Levin06}, where the guide is a single polynomial
and \cite{Karciauskas:2006:CTM} where the guide is allowed to be any
piecewise smooth function, piecewise polynomial in particular.
The critical ingredients for good shape, measured via uniform
curvature and \hld, are a well-shaped guide and its careful sampling.




\section{Definitions and Setup}\label{sec:defset}
Analogous to \CCa-subdivision, we consider the
input a network of quadrilateral facets or \emph{quads}.
Nodes where four quads meet are called regular, otherwise \emph{\eon s}.
Our focus is on the neighborhood of \eon s:
a 2-ring of quads surrounding an \eon\ that we call a \emph{\cnet}.
The 1-ring of  a \cnet\ consists of regular nodes and 
the \cnet\ has $\val\ne 4$ sectors. For example
the interior nodes of \figref{fig:teaser}a form a \cnet.

The subdivision surface will be expressed piecewise
in terms of tensor-product polynomials $\fn$ of
bi-degree $\dg$ in Bernstein-B\'ezier (BB) form
\[
   \fn(u,v):=\sum_{i=0}^\dg\sum_{j=0}^\dg\bz{}{ij}B^\dg_i(u)B^\dg_j(v),
   \quad (u,v) \in \square:=[0..1]^2,
\]
where $B^\dg_k(t) := \binom{\dg}{k}(1-t)^{\dg-k}t^k$
are the Bernstein-B\'ezier (BB) polynomials of degree $\dg$
and $\bz{}{ij}$ are the BB coefficients \cite{Farin02,Prautzsch02}.
A central role will be played by the \emph{corner jet constructor}
\begin{equation*}
   \pc{f}{\dg}{j\times j}{(u_0,v_0)}
\end{equation*}
that expresses the expansion of a function $f$ at $(u_0,v_0)$
up to and including order $j-1$ in $u$ and $j-1$ in $v$
in BB-form of degree bi-$\dg$, i.e.\ 
by $j\times j$ BB-coefficients.
\figref{fig:herm56col}a displays four corner jet constructors
$[f]^5_{3\times 3}$ merged to form a bi-5 patch and 
\figref{fig:herm56col}b displays four corner jet constructors
$[f]^6_{4\times 4}$  merged to form a bi-6 patch by averaging the
overlapping BB-coefficients.
\def\wid{0.3\linewidth}
\begin{figure}[h]
   \centering
   \subfigure[bi-5]{
   \includegraphics[width=\wid]{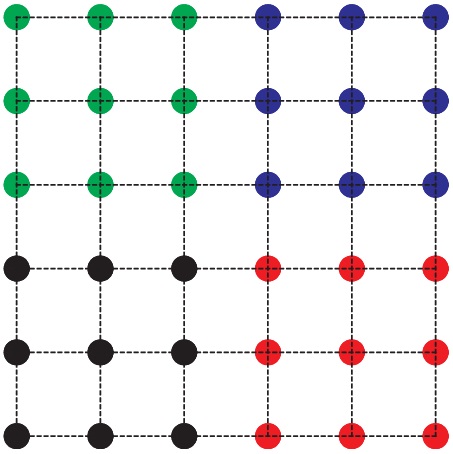}}
   \hskip1.0cm
   \subfigure[bi-6]{
   \includegraphics[width=\wid]{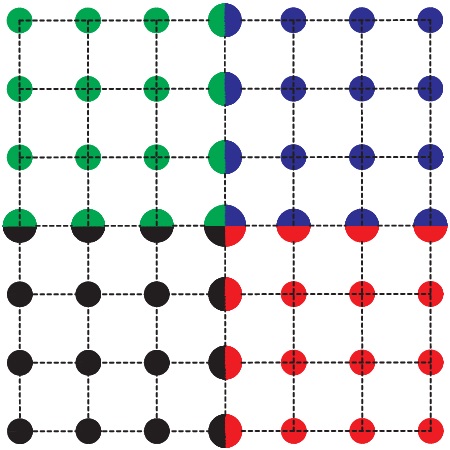}}
   \caption{
   (a) bi-5 patch assembled from $3\times 3$ jets; 
   (b) bi-6 patch assembled by averaging $4\times 4$ jets.   
   }
   \label{fig:herm56col}
\end{figure}

\def\offwid{0.4\linewidth}
\begin{figure}[h]
   \centering
      \epsfig{file=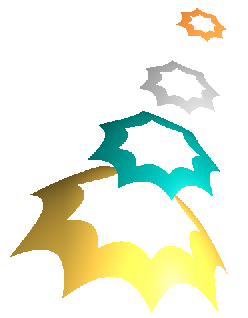,width=\offwid}
      \epsfig{file=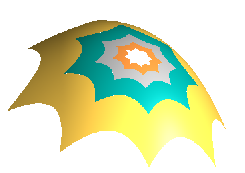,width=\offwid}
   \caption{
   (\IL) Sequence of guided subdivision rings, and
   (\IR) the resulting surface. 
   }
   \label{fig:grings}
\end{figure}

\section{Construction of guided surfaces}\label{sec:gdsurf}
The key to good shape is to construct and judiciously reparameterize
a guide surface so that the jet constructors can form well-shaped
subdivision rings, \chg{i.e.\ a sequence of surface annuli in $\bbr^3$
that join smoothly leaving an ever small central hole (see \figref{fig:grings}).}


\subsection{Guide surface}\label{sec:guide}
In this section, we create a $G^2$ guide surface $\ggg$ of degree bi-5
for filling of multi-sided hole by a series of rings sampled from $\ggg$.

Denote by $L$ the linear shear that maps a unit 
square to the unit parallelogram with opening angle $\frac{2\pi}{n}$. 
Abbreviating $\ac:=\cos\frac{2\pi}{n}$ set
$\bfg := L$ and let $\bfa$ be its reflection
across the edge $v=0$,  \figref{fig:gdtanen}a.
Then, along the common boundary $v=0$,
\begin{align}
   \partial_v\bfa+\partial_v\bfg-2\ac\partial_u\bfg=0\ .
   \label{g1const}
   \\
   \partial^2_v\bfa-\partial^2_v\bfg+4\ac\partial_u\partial_v\bfg-
   4\ac^2\partial^2_u\bfg=0.
   \label{g2const}
\end{align}
The constraints \eqref{g1const} and \eqref{g2const} 
are `unbiased' in the sense that exchanging $\bfa$ and $\bfg$
does not alter the equations.
For degree bi-5 polynomials $\bfg$ and $\bfa$, the equations \eqref{g1const}
and \eqref{g2const} yield a system of $12$ linear equations 
in the BB-coefficients $\bpg_{ij}$ and $\bpa_{ij}$.
These equations are solved symbolically, leaving as unconstrained
BB-coefficients those marked in \figref{fig:gdtanen}b
by \textcolor{red}{red} or black bullets plus one circled cross.


\begin{figure}[h]
   \centering
   \subfigure[shearing]{
   \begin{overpic}[scale=.25,tics=10]{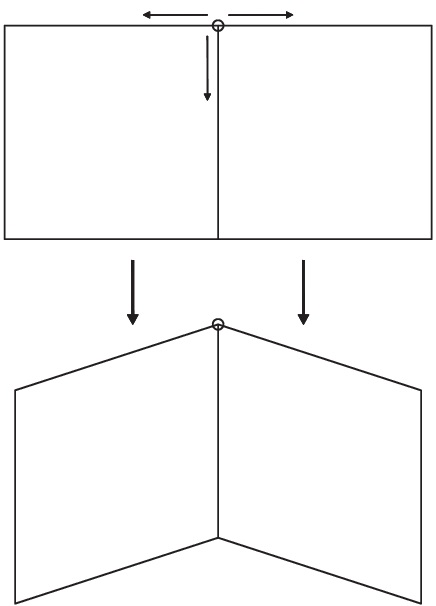}
        \put (20,98) {$v$}
	\put (49,98) {$v$}
	\put (31,80) {$u$}
	\put (16,50) {$L$}
        \end{overpic}}
   \hskip0.1\linewidth
   \subfigure[sector-symmetric index]{
   \begin{overpic}[scale=.3,tics=10]{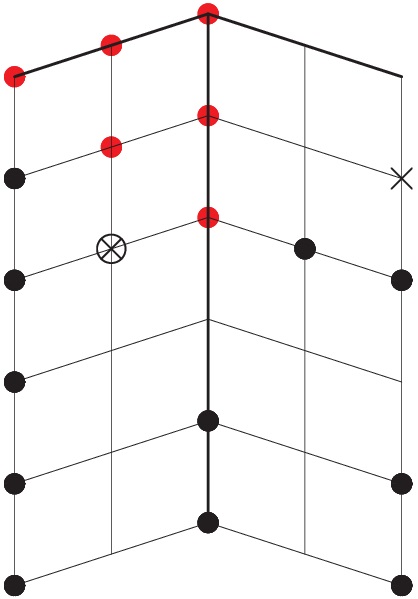}
        \put (28,100) {$00$}
	\put (14,95) {$01$}
	\put (-2,90) {$02$}
	\put (-11,43) {$\bpg:=\bp^k$}
	\put (62,43) {$\bpa:=\bp^{k+1}$}
	\put (49,94) {$01$}
	\put (65,89) {$02$}
	\put (27,82) {$10$}
	\put (-6,54) {$22$}
	\put (69,54) {$22$}
	\put (27,47) {$30$}
	\put (-6,0) {$52$}
	\put (69,0) {$52$}
	\put (31,6) {$50$}
        \end{overpic}}
   \caption{
   (a) Linear shear $L$; (b) bi-5 solution to the local $G^2$ constraints:
   the unconstrained BB-coefficients are marked as \textcolor{red}{red}
   or black bullets or a circled cross.    
   }
   \label{fig:gdtanen}
\end{figure}

The interactions between the $n$ local systems of equations
at the irregular point $\bpg_{00}$ are resolved by
selecting six BB-coefficients $\bp^k_{ij}$, $0\le i+j\le 2$ for $k=0$
and so defining a quadratic expansion at the central point;
the BB-coefficients $\bp^k_{ij}$, $0\le i+j\le 2$, for $k>0$,
are then defined recursively as
\begin{equation}
\begin{aligned}
\bpa_{00}:=& \bpg_{00}\ ,\quad \bpa_{10}:=\bpg_{10}\ ,\quad \bpa_{20}:=\bpg_{20}\ ;\\
\bpa_{01}:=& -\bpg_{01}+2\ac \bpg_{10}+2(1-\ac)\bpg_{00}\ ;\\
\bpa_{11}:=& -\bpg_{11}+\frac{8\ac}{5}\bpg_{20}+(2-\frac{6\ac}{5})\bpg_{10}-\frac{2\ac}{5}\bpg_{00}\ ;\\
\bpa_{02}:=& \bpg_{02}-5\ac\bpg_{11}+4\ac^2\bpg_{20}+(5\ac-4)\bpg_{01}+\ac(9-8\ac)\bpg_{10}\\
& +(4-9\ac+4\ac^2)\bpg_{00}\ . 
\end{aligned}
\label{cenquad}
\end{equation}
Assigning, in each local system,
\begin{align}
   \bpa_{12}:=& 
   \bpg_{12}-2\ac\bpg_{21}+2\ac\bpa_{21}+\ac\bpg_{01}+(3\ac-4)\bpg_{11}
   +\frac{4}{5}\ac(4-3\ac)\bpg_{20}
   \notag
   \\
   & +(4-\frac{27}{5}\ac+\frac{4}{5}\ac^2)\bpg_{10}
   +\frac{c}{5}(8\ac-9)\bpg_{00},
   \label{blackcross}
\end{align}
yields a circulant system of $n$ linear equations in 
$\bp^k_{12}$ (since $\bp^k_{12}=\bpg_{21}$, $\bp^{k+1}_{12}=\bpa_{12}$;
see rotationally-symmetric indexing in \figref{fig:shearglob}). This
system has a unique solution (unless $n=3,6$ when one $\bp^k_{12}$,
say $\bp^0_{12}$, is additionally free to choose).
The explicit expressions for 
$\bpg_{30}$, $\bpg_{31}$, $\bpa_{31}$, $\bpa_{32}$, $\bpg_{41}$, $\bpa_{41}$,
$\bpg_{51}$, $\bpa_{51}$ are presented in Appendix A.

\begin{figure}
        \centering
	\begin{overpic}[scale=.35,tics=10]{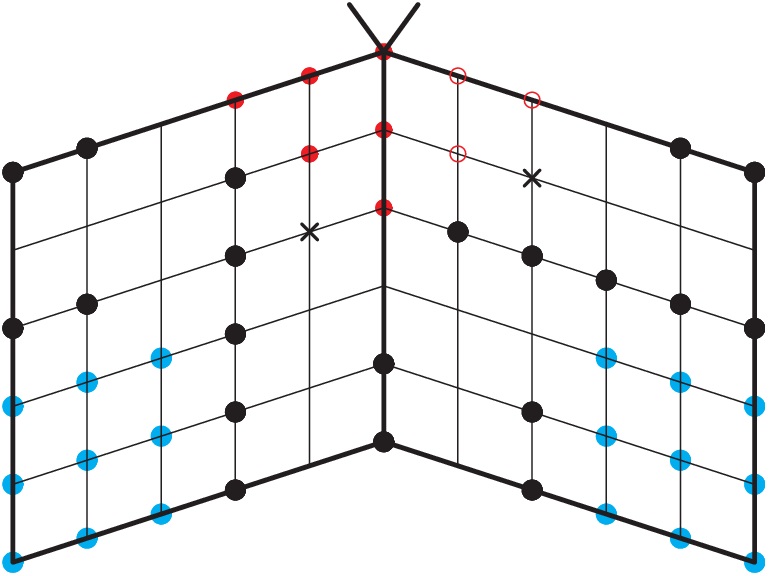}
        \put (47,72) {$00$}
	\put (37,67) {$10$}
	\put (27,64) {$20$}
	\put (43,48) {$02$}
	\put (51,48) {$20$}
	\put (58,67) {$01$}
	\put (68,64) {$02$}
	\put (-6,0) {$55$}
	\put (100,0) {$55$}
	\put (28,5) {$25$}
	\put (100,32) {$25$}
	\put (33,45) {$12$}
	\put (70,52) {$12$}
	\put (-6,22) {$53$}
	\put (75,2) {$53$}
	\put (-7,40) {$\bp^k$}
	\put (101,40) {$\bp^{k+1}$}
        \end{overpic}
        \caption{Rotationally-symmetric indexing of the set $\gfree$ of
        the $G^2$ guide $\ggg$.
        BB-coefficients marked as \textcolor{cyan}{cyan} disks 
        are unconstrained by local $G^2$ requirements.
	}
        \label{fig:shearglob}
\end{figure}

Denote by $\gfree$ the set of BB-coefficients unconstrained by 
the $G^2$ constraints.
We define $n^* :=  n+1$ whenever the valence $n \in \{3,6\}$
and $n^*=n$ otherwise.
The elements of $\gfree$ are shown in \figref{fig:shearglob}:
six \textcolor{red}{red} bullets for the quadratic expansion,
$8n+n^*$ black bullets, and, in each sector corner,
$3\times 3$ \textcolor{cyan}{cyan} bullets that
do not affect $G^2$ continuity between sectors.
The elements of $\gfree$ are pinned down to
best approximate the bi-5 surface $\GGG$ of the \cnet\
as defined in \cite{Kestutis:2015:ISM}.
Let $\gs: \bbr^{2\times n} \to \bbr^2$ be the $2\pi/n$-rotationally
symmetric map obtained by applying the algorithm 
of \cite{Kestutis:2015:ISM} 
to the control points of the planar characteristic \cnet\ of \CCa-subdivision
shown in \figref{fig:gdfit}a.
\figref{fig:gdfit}b, \emph{left}, shows one sector of $\gs$ for $n=5$ 
and 
\figref{fig:gdfit}b, \emph{right} shows $L^{-1}\circ\gs$ to be used for 
sampling the guide $\ggg$.
 
\def\wid{0.2\linewidth}
\def\widd{0.22\linewidth}
\def\widk{0.25\linewidth}
\def\widkk{0.4\linewidth}
\begin{figure}[h]
   \centering
   \subfigure[\cnet]{
   \includegraphics[width=\wid,height=\wid]{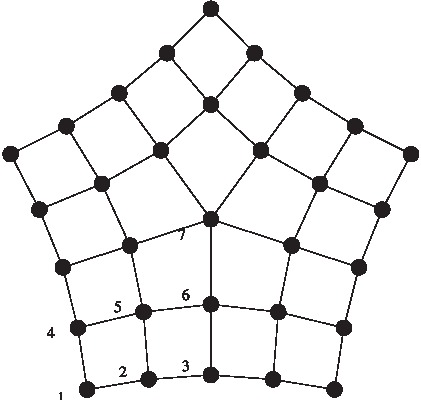}
   }
   \hskip0.5cm
   \subfigure[a sector of $\gs$ and its image under $L^{-1}$]{
   \includegraphics[width=\wid]{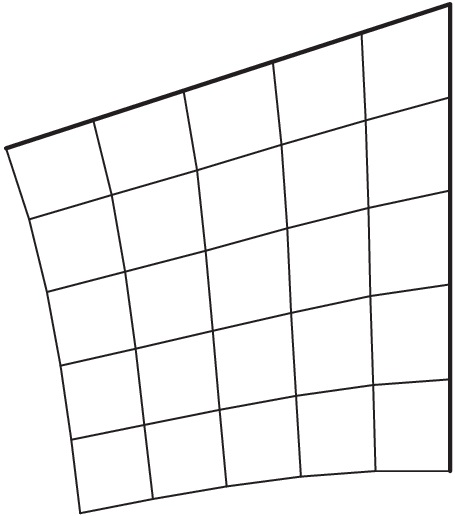}
   \hskip1.0cm
   \includegraphics[width=\widd]{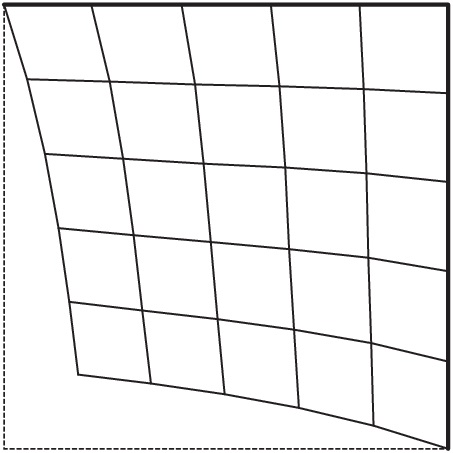}}
   \\
   \subfigure[surrounding surface and $\GGG$]{
   \includegraphics[width=\widkk]{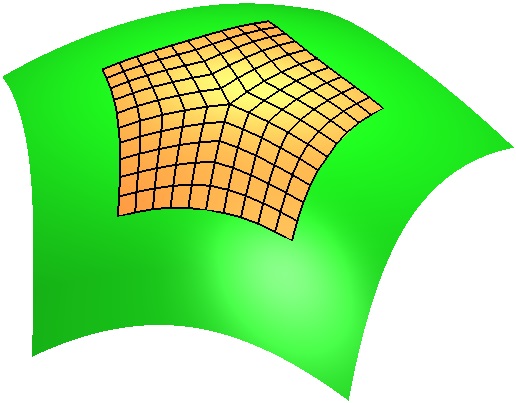}}
   \hskip1.0cm
   \begin{minipage}[b]{\widk}
   \vskip0.05cm
   \subfigure[bi-5 guide $\ggg$]{
   \includegraphics[width=\linewidth]{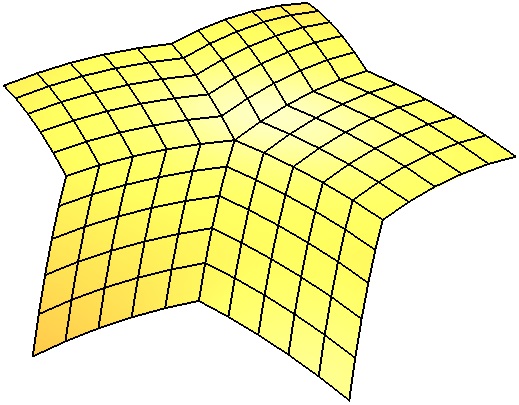}}
   \vskip0.22cm
   \phantom{.}
   \end{minipage}
   \caption{
   {Construction of the guide surface.}
   (c) 
   the surrounding \textcolor{green}{green} bi-3 surface and 
   in gold the bi-5 surface $\GGG$ of \cite{Kestutis:2015:ISM}.
   (d) The net of control points $\ggg^k_{ij}$  (not of final $\bb^k_{ij}$).
   }
   \label{fig:gdfit}
\end{figure}

Setting the central point of $\ggg$ to the central point of $\GGG$, 
we collect the $\textcolor{red}{5}+8n+n^*+\textcolor{cyan}{9n}$=$17n+n^*+5$ 
unconstrained BB-coefficients of the set $\gfree$.
Sampling the $k$th sector according to
$[\ggg^k\circ L^{-1}\circ\gs]^5_{3\times 3}$ 
at all four corners of each unit square sub-domain of $\gs$, 
we obtain the BB-coefficients $\bb^k_{ij}$,
$k=0,\ldots,n-1$, $i,j=0,\ldots,5$. The $\bb^k_{ij}$ are linear
expressions in $\gfree$ and we assemble them to form $n$ patches
$\bb^k$ of degree bi-5. Then $\gfree$ is the least-squares solution of 
minimizing the sum of differences between each $\bb^k_{ij}$ 
and its corresponding coefficient of $\GGG$.

\subsection{Parameterization and guided surface rings}\label{sec:gdrings}

The guide $\ggg$ was constructed to facilitate filling
the multi-sided hole by a contracting sequence of sampled guided 
rings.
The construction, see \figref{fig:gdsampnew},
leverages the fact that the characteristic ring $\charm$
of \CCa\ subdivision joins $C^2$ to its scaled copy $\gl\ \charm$, 
where $\gl:=\frac{1}{16}\big(\ac+5+\sqrt{(\ac+1)(\ac+9)}\big)$
is the subdominant eigenvalue of \CCa\ subdivision.
Since the shear $L$ is linear, also $L^{-1}\circ\charm$ is $C^2$-connected 
to its scaled copy $\gl(L^{-1}\circ\charm)$
(see \figref{fig:gdsampnew}d).
Moreover, the outer second-order Hermite data of $L^{-1}\circ\charm$ 
(underlaid gray in \figref{fig:gdsampnew}a) is determined by 
the $C^2$ prolongation $L^{-1}\circ(\lambda^{-1}\charm)$ and
binary splitting  (see \figref{fig:gdsampnew}c).
\def\wid{0.3\linewidth}
\begin{figure}[h]
   \centering
   \subfigure[image of $\charm$ under $L^{-1}$]{
   \includegraphics[width=\wid]{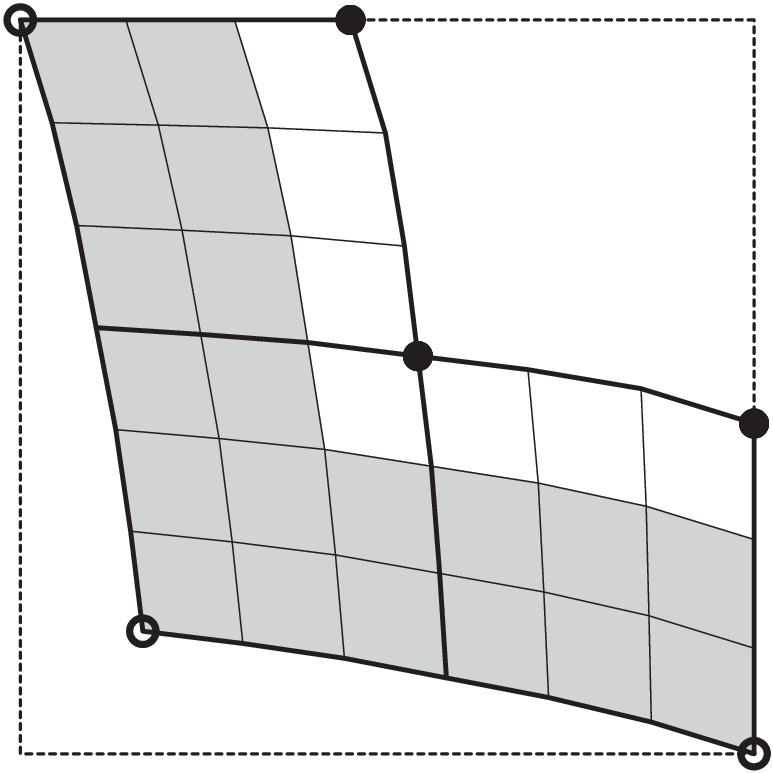}}
   \hskip0.5cm
   \subfigure[$L^{-1}\circ \charmt$]{
   \includegraphics[width=\wid]{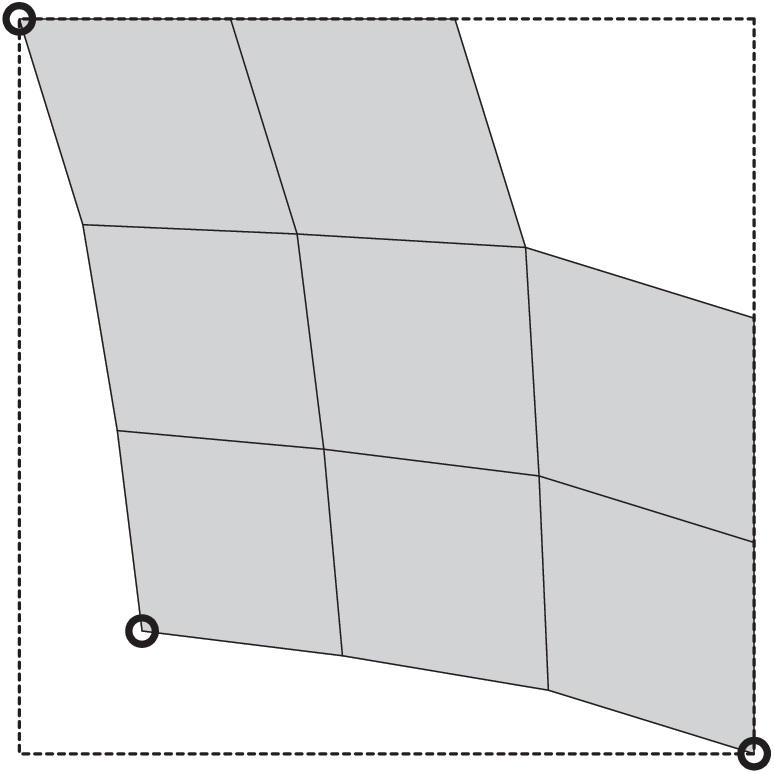}}\\
   \subfigure[sampled bi5 sector]{
   \includegraphics[width=\wid]{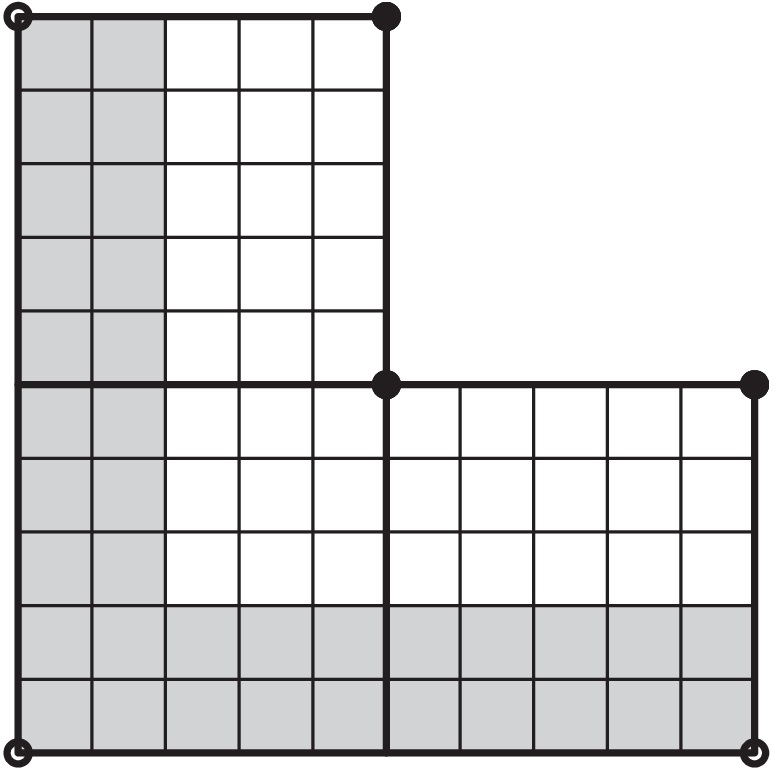}}
   \hskip0.5cm
   \subfigure[$L^{-1}\circ\charm$ scaled by $\gl$]{
   \includegraphics[width=\wid]{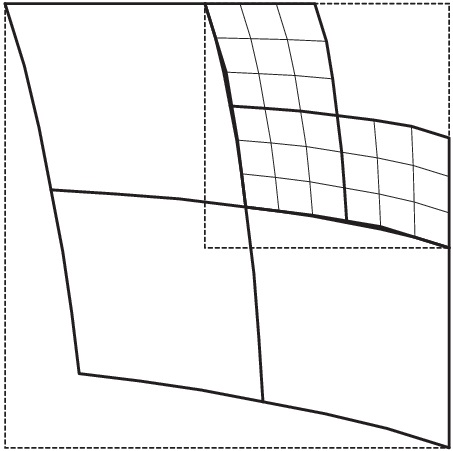}}
   \caption{
   \textbf{Construction of guided ring.}
   }
   \label{fig:gdsampnew}
\end{figure}

\noindent
\textbf{Subdivision Algorithm.}\\
Input: a sector $\ggg^k$ of guide $\ggg$.\\
Output: BB-coefficients of three bi-$d$ ($d=5$ or $d=6$) patches
in the $k$-th sector of the $l$th subdivision ring $\xxb^\llk$, 
$\llk=0,\ldots$.\\
Bi-5 Algorithm:
sample $[\ggg^k\circ(L^{-1}\circ\gl^\llk\charm)]^5_{3\times 3}$ 
at the inner corners marked by bullets in \figref{fig:gdsampnew}a 
to obtain the bi-5 coefficients of \figref{fig:gdsampnew}c
underlaid white;
sample $[\ggg^k\circ(\gl^\llk\charmt)]^5_{3\times 3}$ 
at the outer corners marked as by circles in \figref{fig:gdsampnew}b 
and mid-split the resulting bi-5 data 
to obtain the bi-5 coefficients of \figref{fig:gdsampnew}c
underlaid gray.
The first of these gray sets of BB-coefficients ($\llk=0$) is replaced by 
$C^2$ prolongation from the surrounding surface.

\noindent
The bi-6 algorithm is identical, except that 
$[\ggg^k\circ(L^{-1}\circ\gl^\llk\charm)]^6_{4\times 4}$ is
sampled and assembled according to \figref{fig:herm56col}b.

\begin{prop}
The Subdivision Algorithm fills a multi-sided hole in 
a spline surface $C^2$ up to the central point.
If the degree is bi-6, the surface is $C^2$;
if the degree is bi-5, the surface is $C^2$ 
except for the central point where it is $C^1$ and curvature bounded.
\end{prop}
\begin{proof}
Since $\charm$ is $C^2$ and 
\chg{
since the same $G^1$ and $G^2$ constraints \eqref{g1const} and \eqref{g2const}
tie together adjacent sectors of $\ggg$ and tie together adjacent sectors of 
$L$,}
adjacent sectors are $C^2$-connected. Due to prolongation, 
adjacent rings are $C^2$-connected and so is the first ring to 
the surrounding spline surface. The limit point characterization
follows by the arguments of 
\cite[Thm 1]{Karciauskas:2006:CTM} applied to 
$(\ggg\circ L^{-1})\circ\charm$.
\end{proof}

Since the guide $\ggg$ closely approximates the surface
$\GGG$ constructed according to \cite{Kestutis:2015:ISM}, 
the good shape of $\GGG$ is retained.

\noindent
\textbf{Efficient Implementation.} 
Computing $\pc{\ggg^k\circ(L^{-1}\circ\gl\charm)}{5}{3\times 3}{}$ 
for the bi-5 Subdivision Algorithm
is equivalent to (i) linearly mapping $S:[0..1]^2 \to [0..\gl]^2$ 
and (ii) sampling 
$\pc{(\ggg^k\circ S)\circ(L^{-1}\circ\charm)}{5}{3\times 3}{}$ and
$\pc{(\ggg^k\circ S)\circ(L^{-1}\circ\charmt)}{5}{3\times 3}{}$.
In Step (i) applying de Casteljau's algorithm at $u=\gl=v$
yields the BB-coefficients of $\ggg^k\circ S$ as an affine combination 
of $\ggg^k$. This affine $6^2\times 6^2$ map is tabulated.
For (ii), a sector consisting of three 
patches of degree bi-5 is pre-computed for each valence $n$
as a table of size $3(6^2\times \cdot6^2)$, by sampling 
$\pc{\ggg^k\circ(L^{-1}\circ\charm)}{5}{3\times 3}{}$
and
$\pc{\ggg^k\circ(L^{-1}\circ\charmt)}{5}{3\times 3}{}$
for symbolic $\ggg^k_{ij}$.
Analogously, for bi-6 subdivision, we pre-compute  for each valence $n$
a table of size $3(7^2\times 6^2)$.

If the storage of the pre-computed data is a concern, one can leverage
that the majority of the entries can be derived by $C^2$ prolongation from
the previous ring and the remainder is defined by two jets 
in each sector, one of which has diagonal symmetry.
For bi-5 subdivision, this reduces the tabulatioa \chg{for each valence}
by a factor of $\sim7$
from $3(6^2\times 6^2)$ to $15*6^2$.

\chg{
Computing
$\pc{\ggg^k\circ(L^{-1}\circ\gl\charm)}{6}{4\times 4}{}$
analogously yields an efficient implementation for the 
bi-6 variant.}




\def\wid{0.3\linewidth}
\begin{figure}[h]
   \centering
   \subfigure[bi-6]{
   \includegraphics[width=\wid]{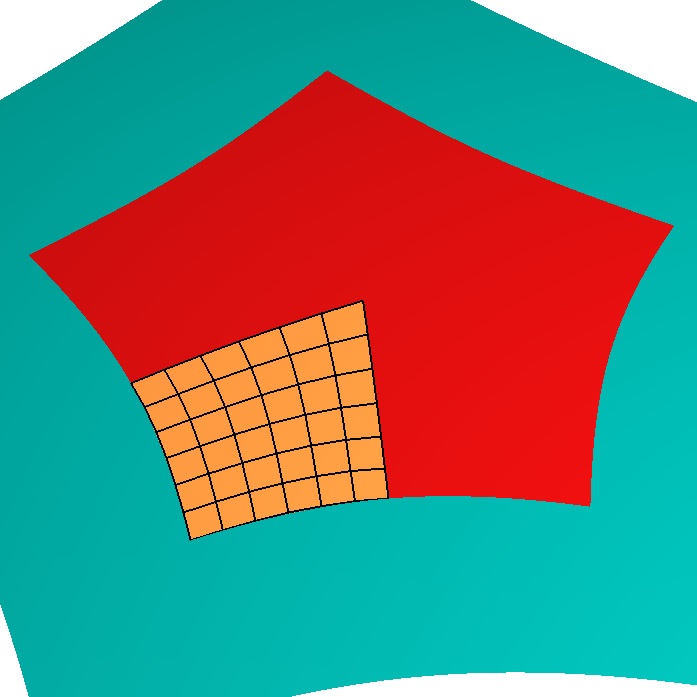}}
   \hskip1.0cm
   \subfigure[bi-5]{
   \includegraphics[width=\wid]{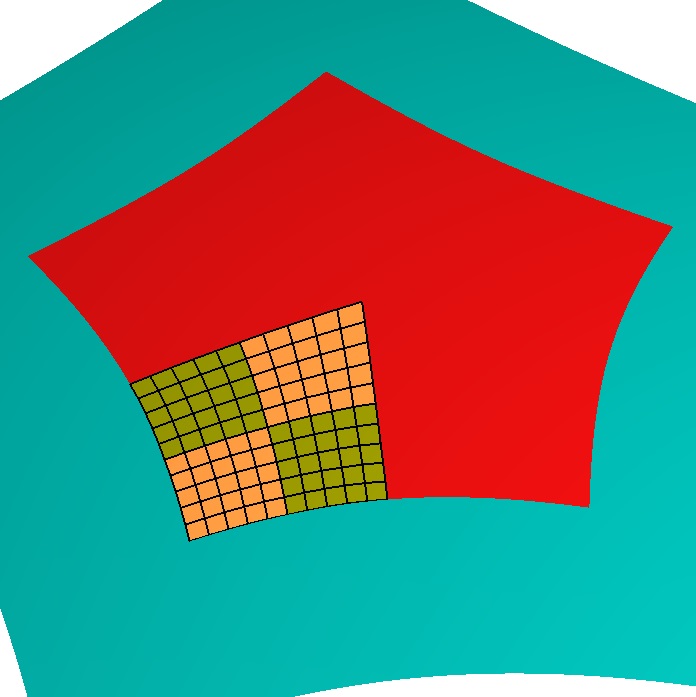}}
   \caption{
   $G^1$ caps for $n=5$. The cap is surrounded
   by the last ring produced by guided
   subdivision (blue).  
   }
   \label{fig:bi56caps}
\end{figure}

\noindent
\textbf{A Practical Hybrid.} 
Being able to cap the subdivision surface after a few steps
is useful in practice, already for visualization (see the next section).
For \CCa\ subdivision the first refinement steps introduce 
\hl\ distortions that no cap can repair.
By contrast, guided subdivision retains good shape and
enables capping without loss of shape quality.
\figref{fig:bi56caps} shows the natural structure of 
these bi-6 and bi-5 caps (bi-5 caps with one patch per sector
lead to reduced quality, hence the split). 
A prototypical construction, of a bi-6 cap,
is detailed in Appendix B.
The good shape of this finite hybrid representation makes
them aesthetically useful in their own right and
the finite number of polynomial surface pieces simplifies
their use for downstream applications and to serve
as a domain for surface-based computations.

\def\wid{0.2\linewidth}
\begin{figure}[h]
   \centering
   \subfigure[$n=5$]{
   \includegraphics[width=\wid]{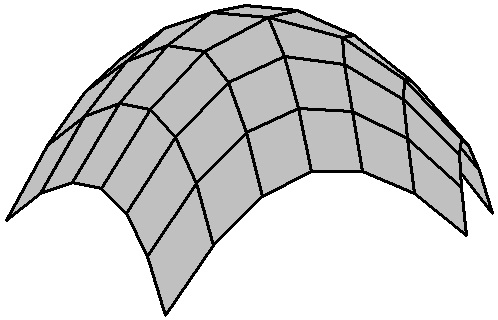}}
   \hskip0.2cm
   \subfigure[$n=5$]{
   \includegraphics[width=\wid]{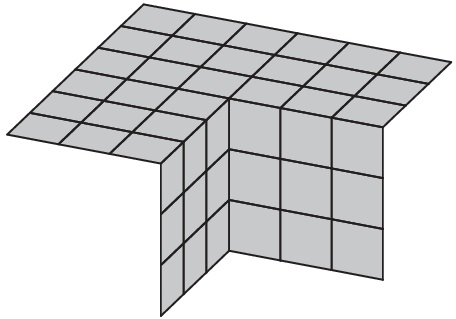}}
   \hskip0.2cm
   \subfigure[$n=6$]{
   \includegraphics[width=\wid]{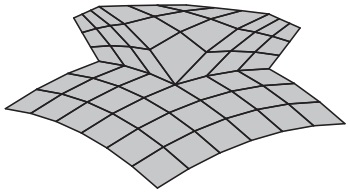}}
   \hskip0.2cm
   \subfigure[$n=7$]{
   \includegraphics[width=\wid]{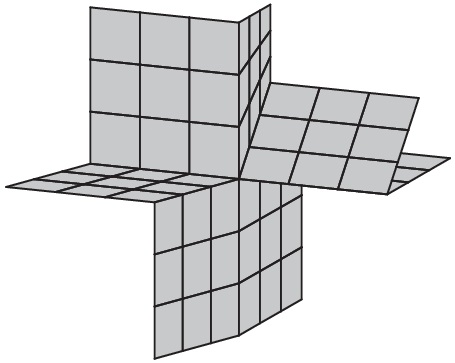}}\\
   \subfigure[$n=8$]{
   \includegraphics[width=\wid]{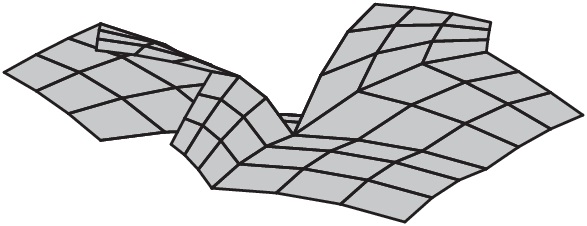}}
   \hskip0.2cm
   \subfigure[$n=9$]{
   \includegraphics[width=\wid]{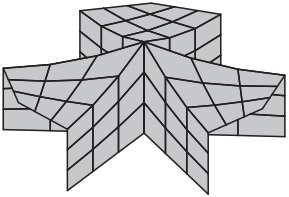}}
   \hskip0.2cm
   \subfigure[$n=3$]{
   \includegraphics[width=\wid]{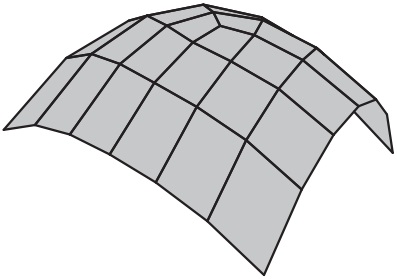}}
   \hskip0.2cm
   \subfigure[$n=6$]{
   \includegraphics[width=\wid]{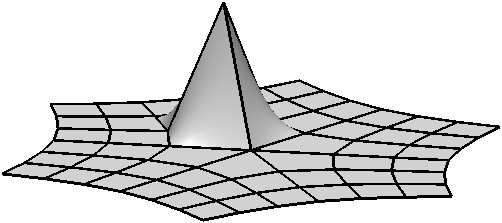}}
   \caption{
   Some challenging $\bc$-nets.
   }
   \label{fig:ccmeshes}
\end{figure}
\def\wid{0.3\linewidth}
\def\skp{\hskip 0.00\linewidth}
\begin{figure}[h]
   \centering
   \subfigure[$4$ bi-5 guided rings]{
   \includegraphics[width=\wid]{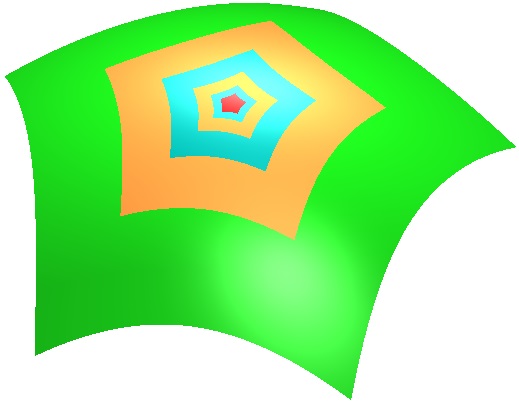}}
   \skp
   \subfigure[\hl s]{
   \includegraphics[width=\wid]{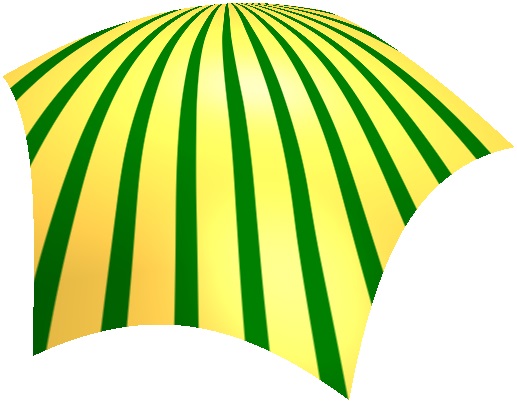}}
   \skp
   \subfigure[Gauss curvature]{
   \includegraphics[width=\wid]{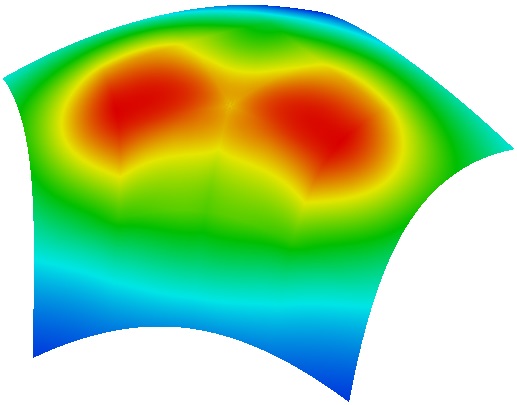}}\\
   \subfigure[$4$ bi-6 guided rings]{
   \includegraphics[width=\wid]{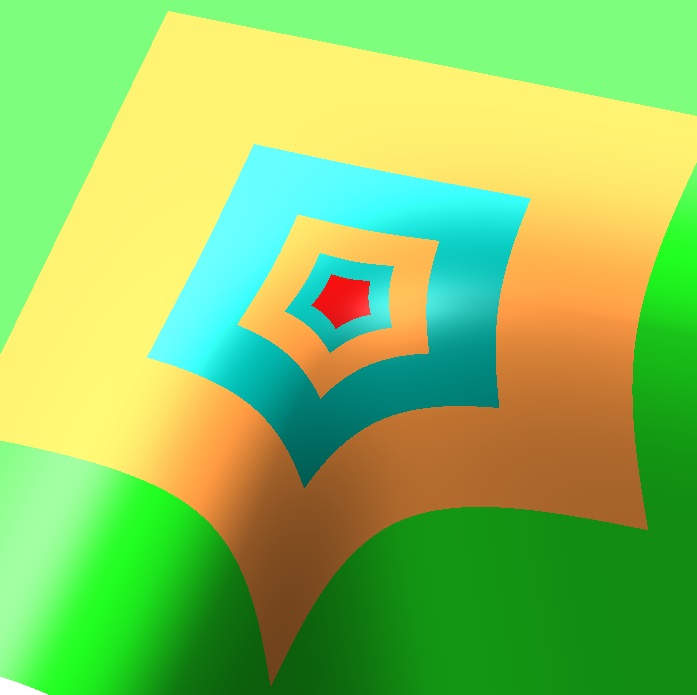}}
   \skp
   \subfigure[$6$ guided rings]{
   \includegraphics[width=\wid]{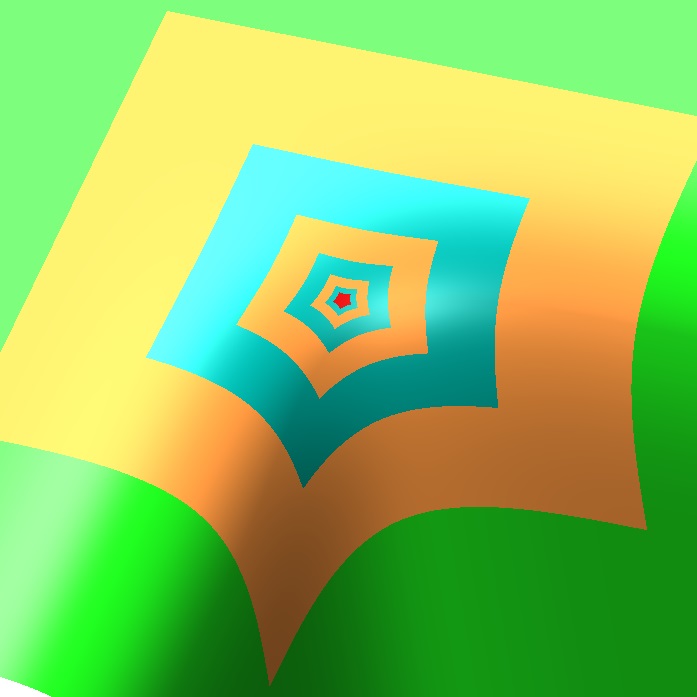}}
   \skp
   \subfigure[\hl s]{
   \includegraphics[width=\wid]{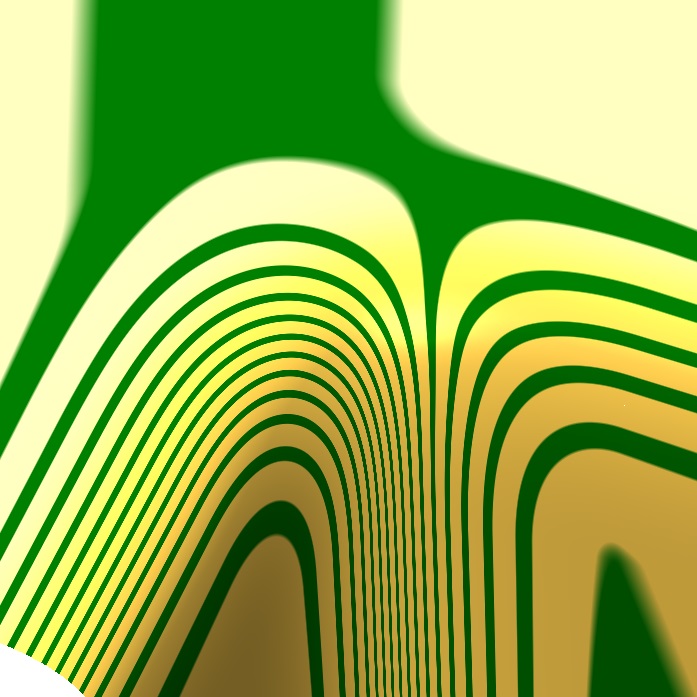}}
   \\
   \subfigure[zoom of (d): highlights and mean curvature]{
   \includegraphics[width=\wid]{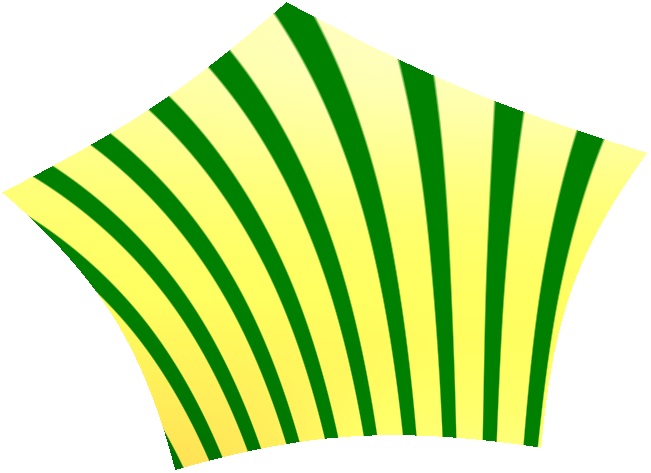}
   \skp
   \includegraphics[width=\wid]{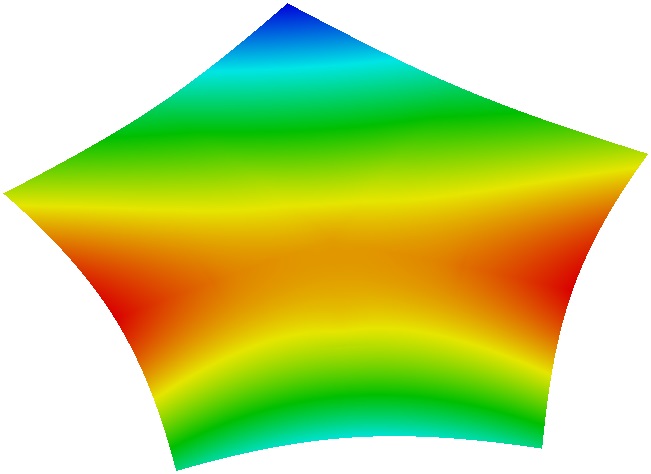}}
   \caption{
   Guided rings and central \textcolor{red}{cap} (in red) for 
   $n=5$: \emph{top} row: \cnet\ of \figref{fig:ccmeshes}a
   \emph{middle, bottom} rows: \cnet\ of \figref{fig:ccmeshes}b.
   }
   \label{fig:n5surf}
\end{figure}
\def\wid{0.22\linewidth}
\begin{figure}[H]
   \centering
   \subfigure[$4$ guided rings of degree bi-6]{
   \includegraphics[width=\wid]{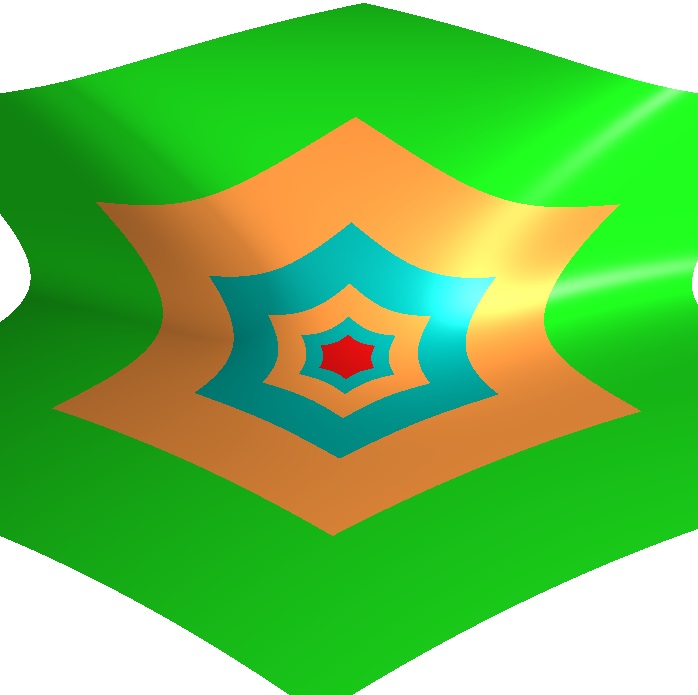}
   \skp
   \includegraphics[width=\wid]{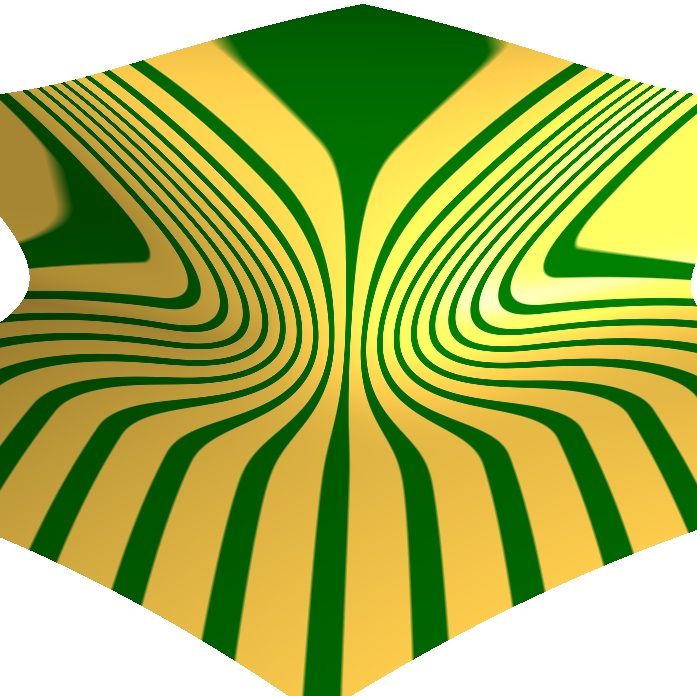}}
   \skp
   \subfigure[zoom; mean curvature]{
   \includegraphics[width=\wid]{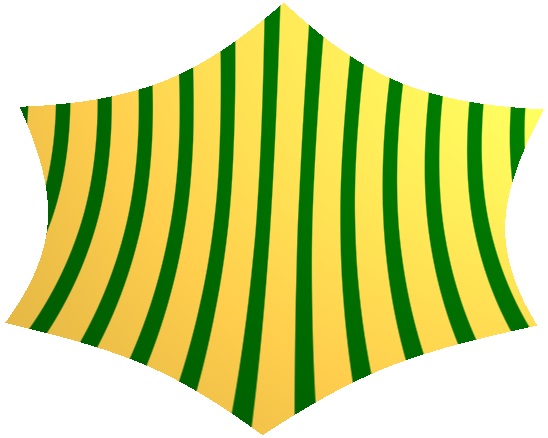}
   \skp
   \includegraphics[width=\wid]{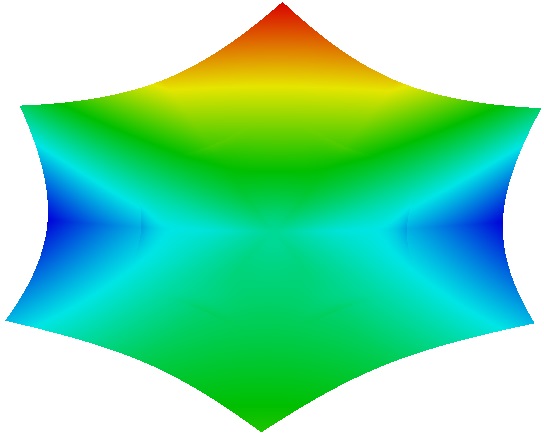}}\\
   \subfigure[$4$ guided rings of degree bi-5]{
   \includegraphics[width=\wid]{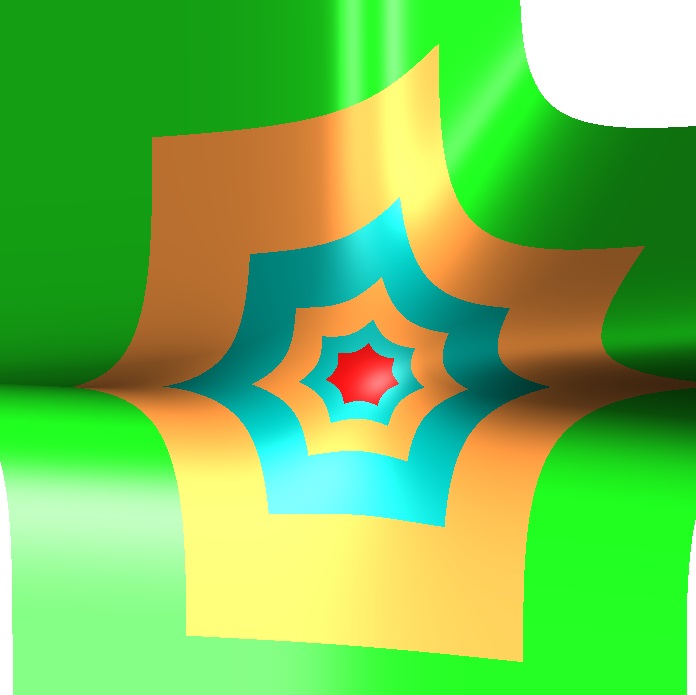}
   \skp
   \includegraphics[width=\wid]{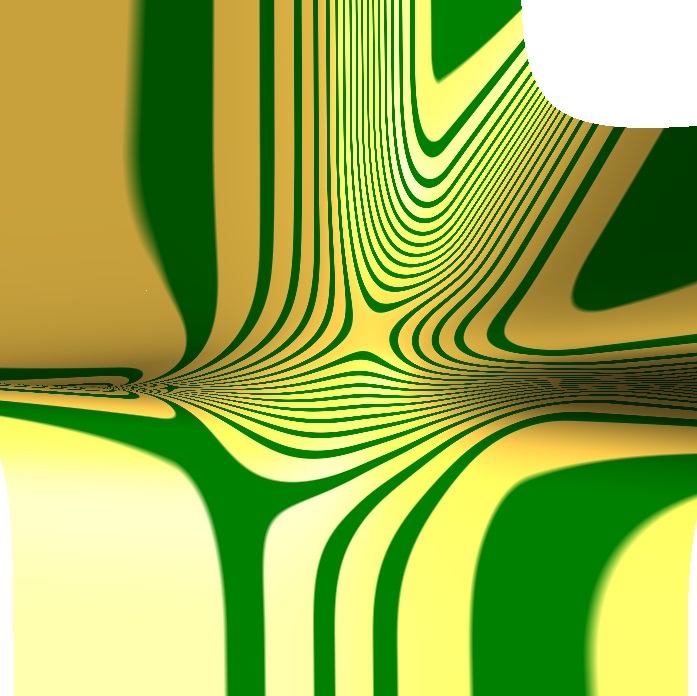}}
   \skp
   \subfigure[zoom; Gauss curvature]{
   \includegraphics[width=\wid]{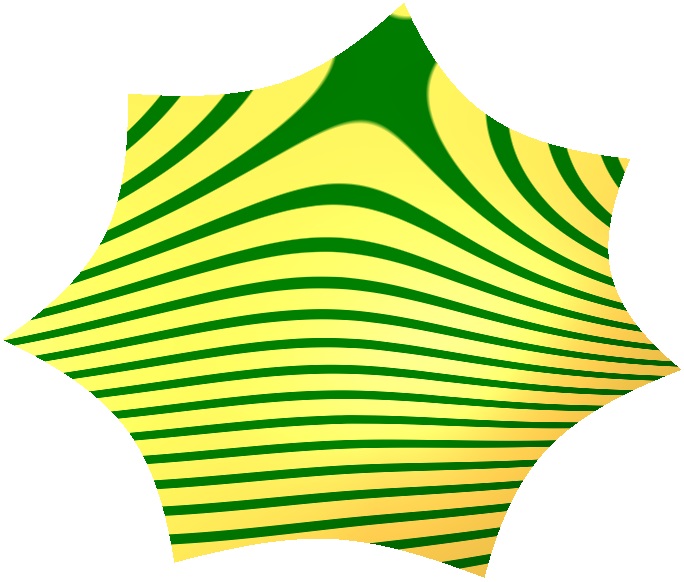}
   \hskip0.1cm
   \includegraphics[width=\wid]{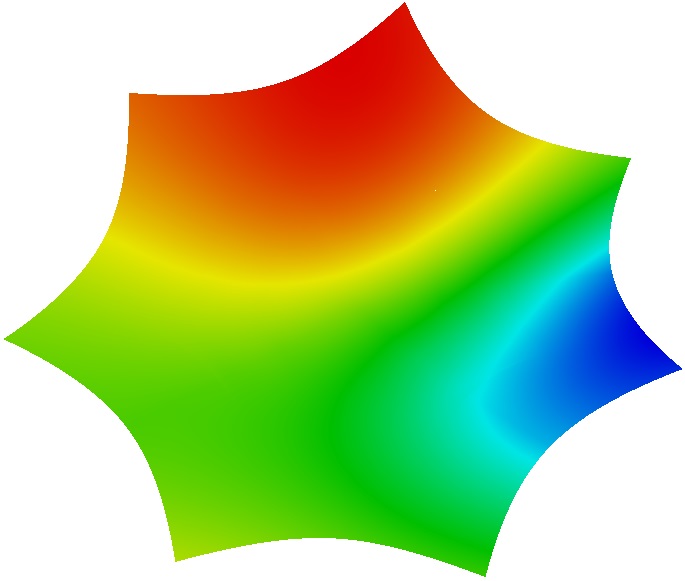}}
   \caption{Guided subdivision rings, \hl s and curvature.
   \IT\ row: bi-6 surfaces, $n=6$, \cnet\ \figref{fig:ccmeshes}c;
   \IB\ row: bi-5 surfaces, $n=7$, \cnet\ \figref{fig:ccmeshes}d.   
   }
   \label{fig:n67surf}
\end{figure}
\section{Examples}\label{sec:examp}
The Subdivision Algorithm is applied to the
challenging \cnet s presented in \figref{fig:ccmeshes}.
Displaying the surrounding spline surface in \textcolor{green}{green}
\chg{
in \figref{fig:n5surf} for valence $n=5$, 
\figref{fig:n67surf} for valences $n=6,7$, 
and 
\figref{fig:n893surf}
for valences $n=8,9$ and $n=3$, }
demonstrates that we do not just create good caps for the
various $n$, but caps that transition well from the input data.
The finite caps of the hybrid representation
are \textcolor{red}{red}. 'Zoom' indicates that innermost
guided ring and the finite cap are displayed. The zoom demonstrates
that even the caps have a calm curvature distribution,
quite in contrast to \CCa-subdivision, where the shape 
deficiencies are acerbated at each step.

\def\wid{0.22\linewidth}
\def\widd{0.25\linewidth}
\begin{figure}[H]
   \centering
   \subfigure[$4$ guided rings, $n=8$, bi-5]{
   \includegraphics[width=\wid]{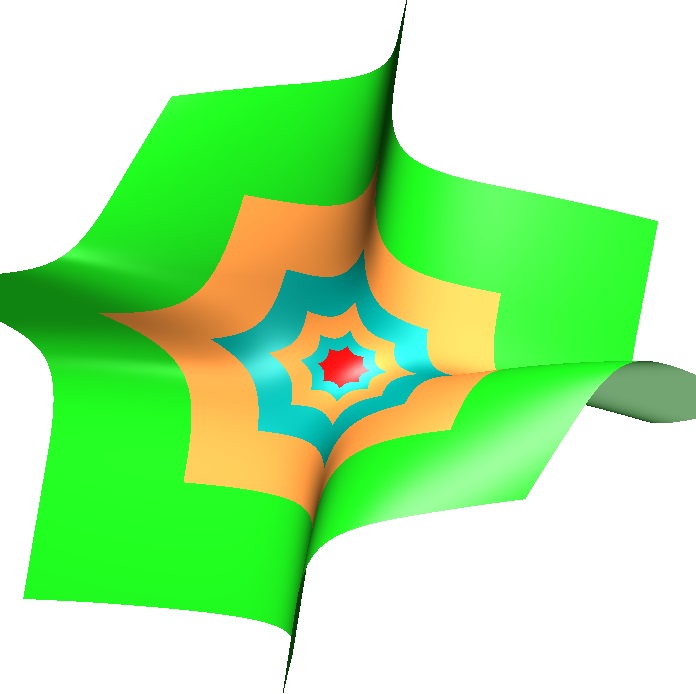}
   \skp
   \includegraphics[width=\wid]{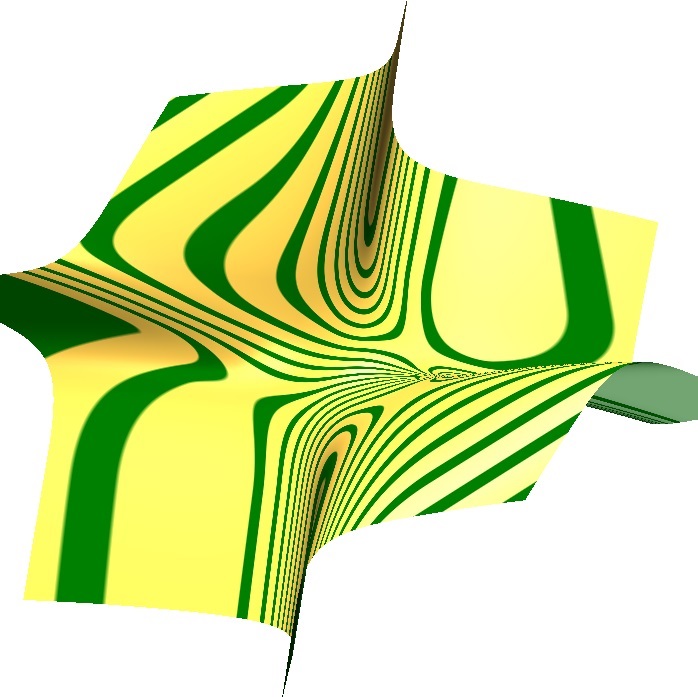}}
   \skp
   \skp
   \subfigure[$5$ guided rings, $n=9$, bi-5]{
   \includegraphics[width=\wid]{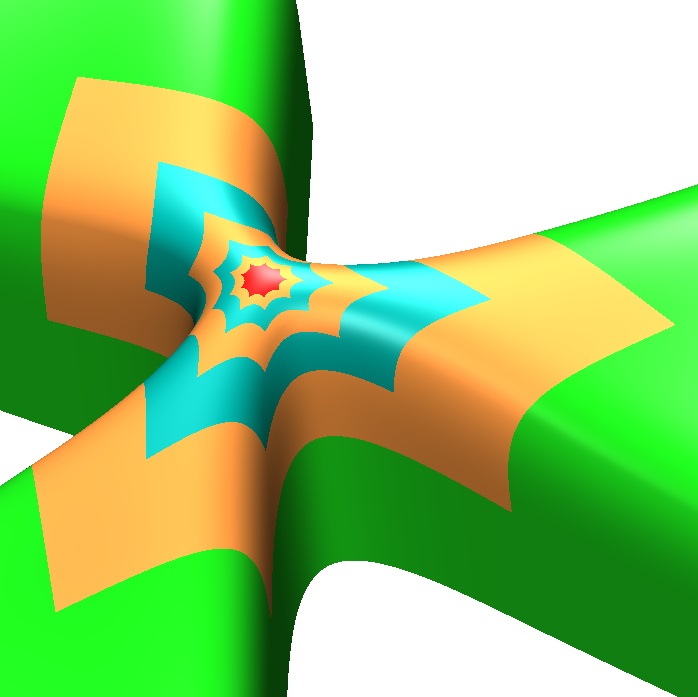}
   \skp
   \includegraphics[width=\wid]{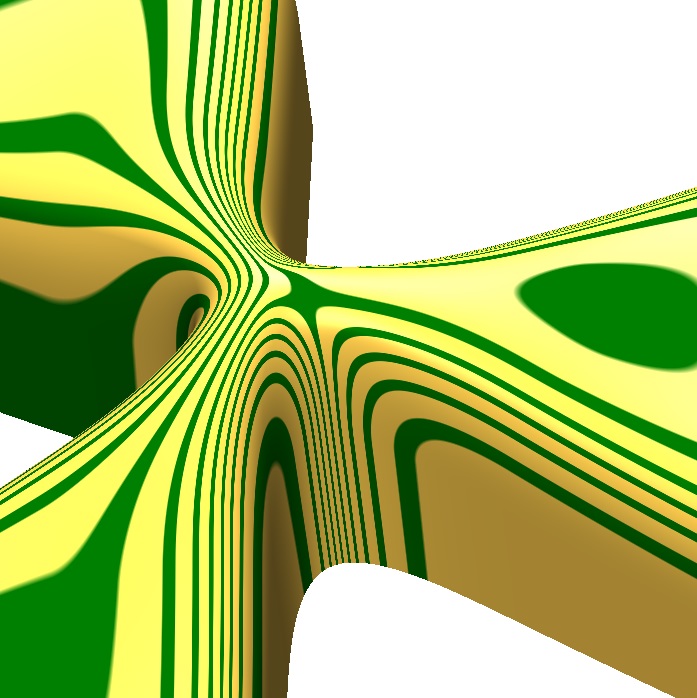}}\\
   \subfigure[$3$ guided rings, $n=3$, bi-5, Gauss curvature]{
   \includegraphics[width=\widd]{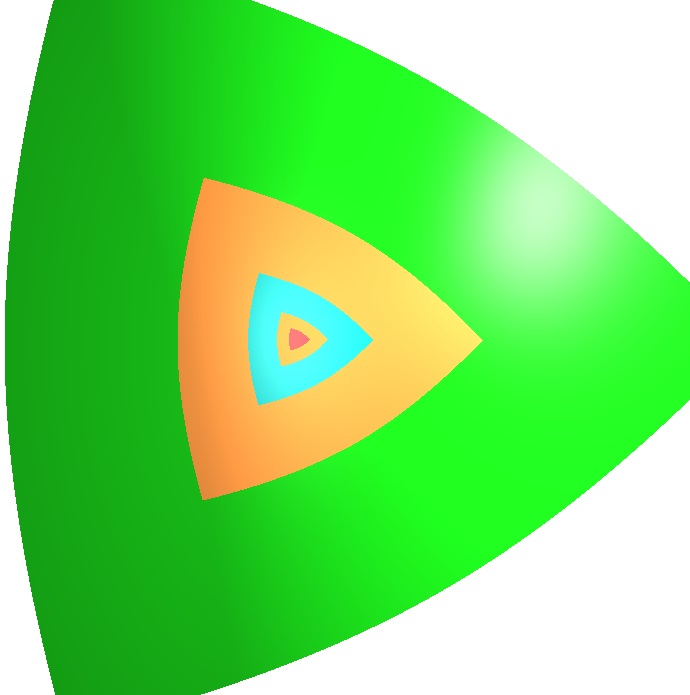}
   \hskip0.1cm
   \includegraphics[width=\widd]{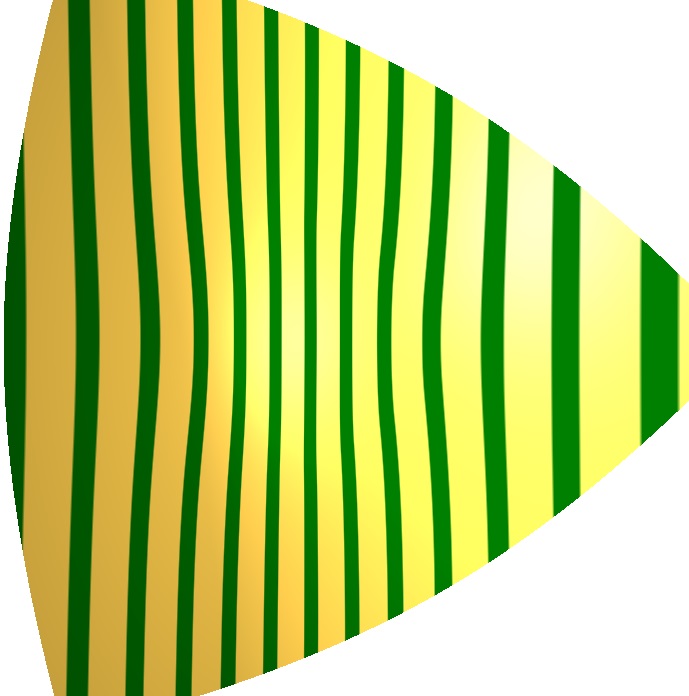}
   \hskip0.1cm
   \includegraphics[width=\widd]{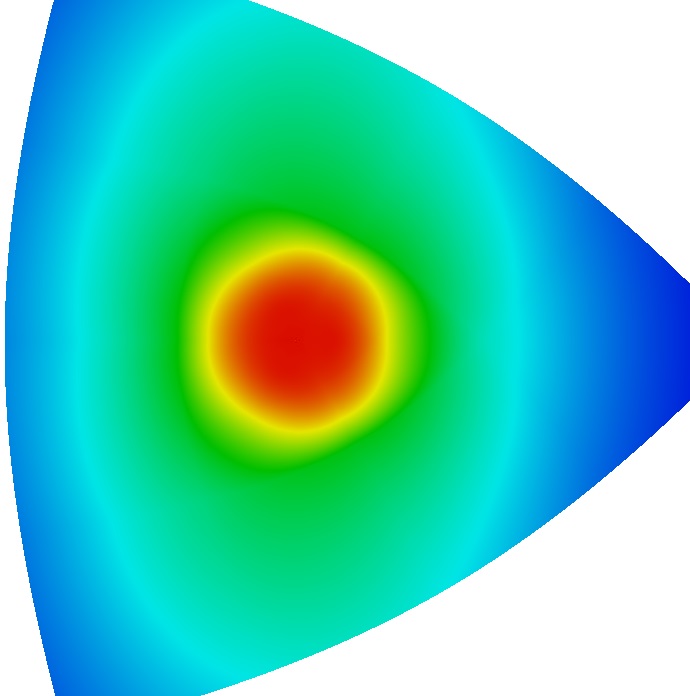}}
   \caption{Bi-5 guided subdivision surfaces for
   (a) $n=8$, \cnet\ \figref{fig:ccmeshes}e; 
   (b) $n=9$, \cnet\ \figref{fig:ccmeshes}f;
   (c) $n=3$, \cnet\ \figref{fig:ccmeshes}g.   
   }
   \label{fig:n893surf}
\end{figure}

\section{Eigen-decomposition}\label{sec:eigencomp}

%
The eigen-decomposition of guided subdivision differs from that of 
conventional subdivision such as \CCa, in that it is defined by the
guide surface: any guided subdivision
inherits the guide's eigen-decomposition.
Therefore the decomposition does not overtly involve  the specification of 
large circulant matrices or of characteristic polynomials;
and the same analysis applies to guided subdivision constructions
of different polynomial degree. 
If we consider just one bi-5 patch of $\ggg$ without neighbor-interaction,
the scaling by $\gl$ to the next-level subdivision ring
(\secref{sec:gdrings}) yields, due to the monomial structure,
eigenvalues $\gl^s$ for total degree $s=0,\ldots,10$  
and eigenfunctions $u^iv^j$ of tensor-degree $(i,j)$, $0\le i,j \le 5$ 
(see  \figref{fig:quadmono}).  
The eigen-analysis of a full ring is more complex due to
the interaction between neighbors, but retains this basic structure.

\begin{figure}[h]
   \centering
   \begin{overpic}[scale=.35,tics=10]{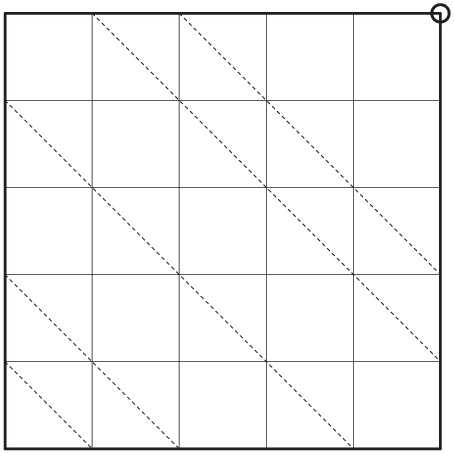}
        \put (100,92) {$1$}
	\put (98,73) {$v$}
	\put (98,54) {$v^2$}
	\put (98,39) {$s=3$}
	\put (98,20) {$s=4$}
	\put (-17,39) {$s=8$}
	\put (-17,20) {$s=9$}
	\put (72,93) {$u$}
	\put (54,93) {$u^2$}
	\put (69,74) {$uv$}
	\put (98,0) {$v^5$}
	\put (-7,93) {$u^5$}
	\put (-14,2) {$u^5v^5$}
	\put (50,39) {$u^2v^3$}
	\put (50,2) {$u^2v^5$}
	\put (10,39) {$u^4v^3$}
	\put (-14,57) {$u^5v^2$}
	\end{overpic}
    \caption{
   Monomial structure of the local eigen-decomposition.
   }
   \label{fig:quadmono}
\end{figure}
The Subdivision Algorithm constructs the $\ell$th surface ring 
from a vector of BB-coefficients $P_\ell$ corresponding to the
set $\gfree$ of the guide $\ggg$ restricted to $[0..\gl^\ell]^2$.
With $M$ the map so that $P_\ell = M P_{\ell-1}$,
the eigen-decomposition determines $\eigv$ and
$\mu$ so that $M \eigv = \mu \eigv$. Appendix C shows that all eigenvalues $\mu$ are of the
form $\gl^s$ where $0 < \gl < 1$  is free to choose:
\begin{align}
\begin{matrix}
   \gl^s, s= & 0 & 1 & 2 & 3 & 4 & 5,6,7,8 & 9 & 10 \\
   \text{multiplicity} & 1 & 2 & 3 & n^* & 2n &  3n& 2n & n 
\end{matrix}
\label{eq:eigvals}
\end{align}
where $n^* =n+1$ when $n \in \{3,6\}$ and $n^*=n$ otherwise.
\def\wid{0.22\linewidth}
\begin{figure}[h]
   \centering
   \subfigure[$\gl$]{
   \includegraphics[width=\wid]{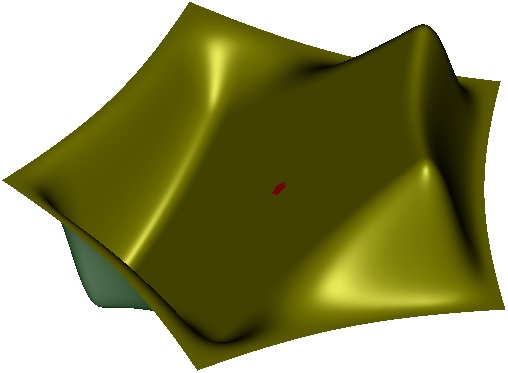}}
   \hskip0.1cm
   \subfigure[$\gl^2$]{
   \includegraphics[width=\wid]{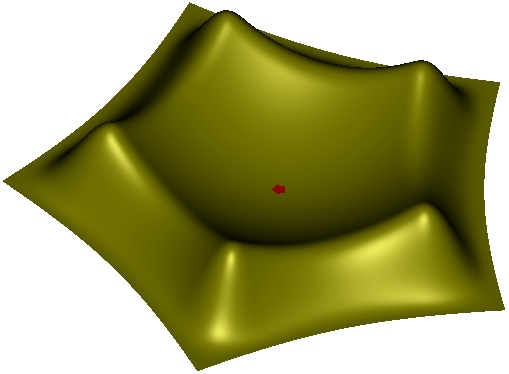}}
   \hskip0.1cm
   \subfigure[$\gl^2$]{
   \includegraphics[width=\wid]{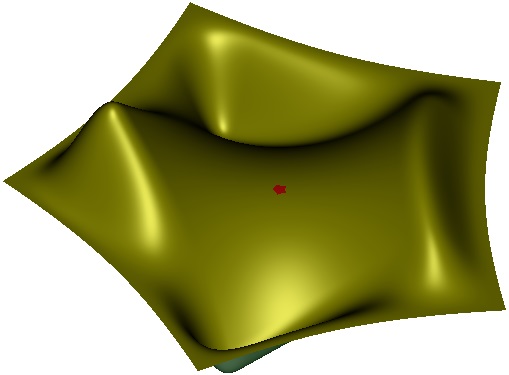}}
   \hskip0.1cm
   \subfigure[$\gl^2$]{
   \includegraphics[width=\wid]{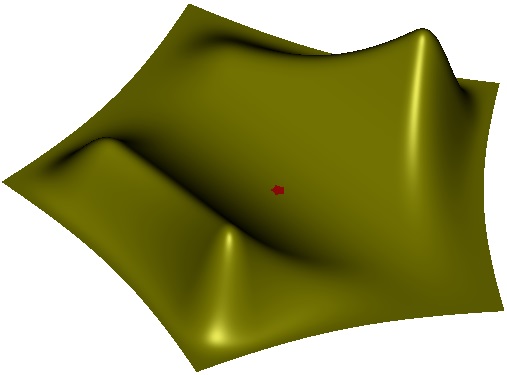}}\\
   \subfigure[$\gl^3$]{
   \includegraphics[width=\wid]{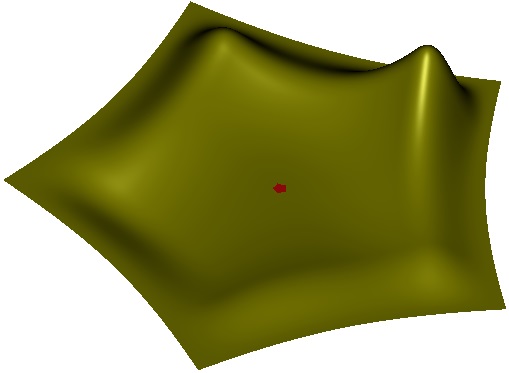}}
   \hskip0.1cm
   \subfigure[$\gl^4$]{
   \includegraphics[width=\wid]{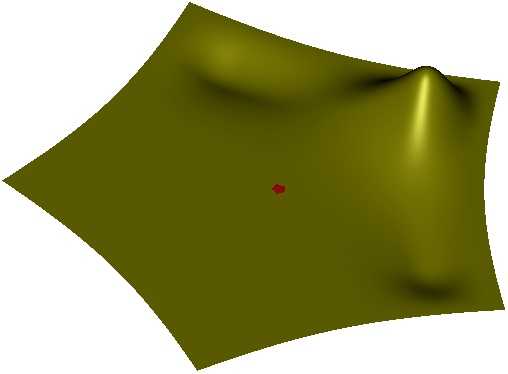}}
   \hskip0.1cm
   \subfigure[$\gl^4$]{
   \includegraphics[width=\wid]{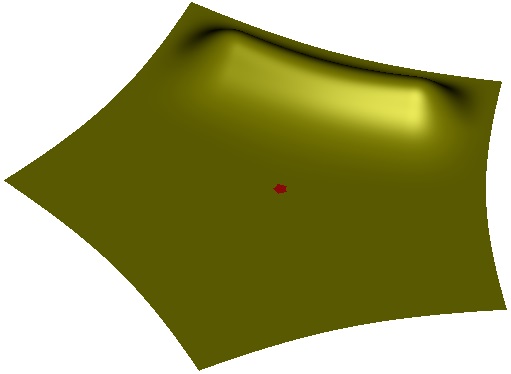}}
   \hskip0.1cm
   \subfigure[$\gl^5$]{
   \includegraphics[width=\wid]{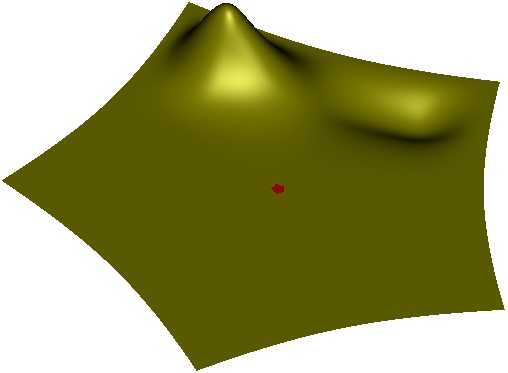}}\\
   \subfigure[$\gl^6$]{
   \includegraphics[width=\wid]{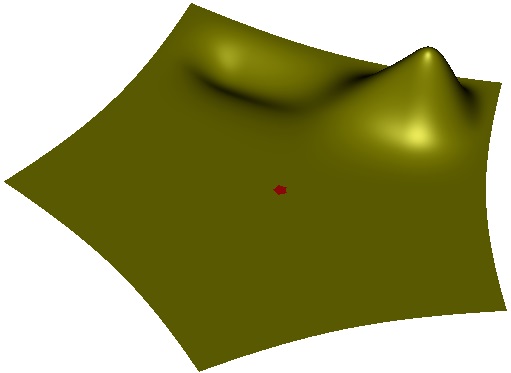}}
   \hskip0.1cm
   \subfigure[$\gl^7$]{
   \includegraphics[width=\wid]{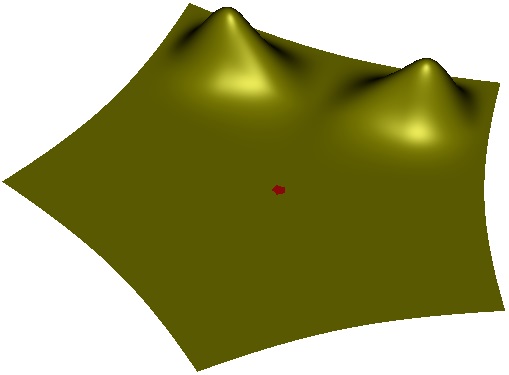}}
   \hskip0.1cm
   \subfigure[$\gl^8$]{
   \includegraphics[width=\wid]{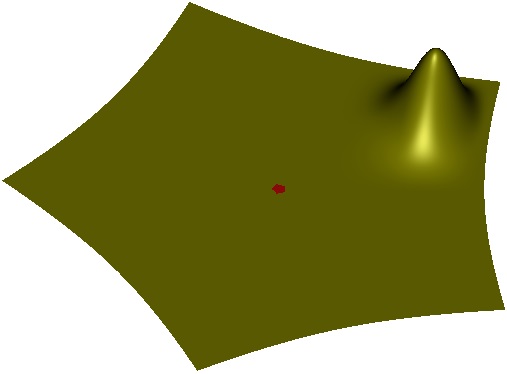}}
   \hskip0.1cm
   \subfigure[$\gl^0 =1$]{
   \includegraphics[width=\wid]{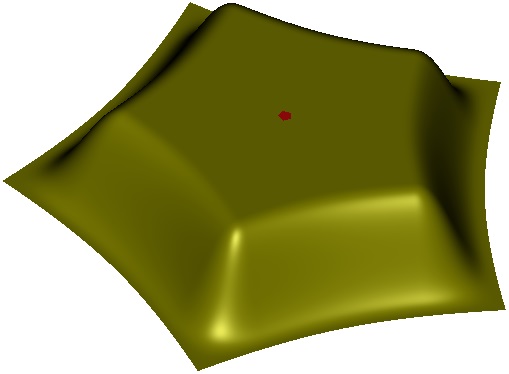}}
   \caption{
   Some bi-5 eigenfunctions for $n=5$.   
   }
   \label{fig:eig5func}
\end{figure}
The bi-5 eigenfunctions for $n=5$ shown in \figref{fig:eig5func} 
are determined by setting one free parameter $x^s_k$
(see \figref{fig:eigenstr}) to $1$
and all others to $0$, completing this determining set to 
a smooth eigenring according to the Subdivision Algorithm,
and generating six scaled copies before filling the center with
a (red) cap. 
For $\gl^2$
the functions in \figref{fig:eig5func} are scaled in height and,
at each level, we formed, 
linear combinations of the eigenrings to show the typical
cup and two saddle shapes of the quadratic terms.
After assembling the surface, the outermost BB-coefficients 
were set to zero.

Due to the index-rotational symmetry, for $s=3,\ldots10$
only $23$ $(=17+1+5)$ eigenfunctions
(besides the constant one for the eigenvalue 1)
need to be tabulated. An exception to this symmetry is 
$s=3$ when $n =3,6$ and $26,29$ functions are required.
Since the eigen-ring $\ell$ is obtained by multiplying the initial
eigen-ring by $(\gl^s)^\ell$, it suffices to store the initial eigen-ring.
The $\ell$th surface ring 
is a linear combination of these $(\gl^s)^\ell$-scaled eigen-rings.

\section{Discussion}\label{sec:disc}
Here we discuss alternatives and options that were 
not given prominence in the main construction:
reducing the degree of the guide (to simplify the
eigenstructure), reducing the degree of the overall
guided subdivision scheme, changing the speed of the 
contraction of the subdivision rings and refining functions on
a \gss.

\subsection{Guides of lower degree}\label{sec:lowguide}

The degree of the subdivision surfaces is not strongly coupled to the 
degree of the guide surface. We can choose guides of degree bi-4 or bi-3
where the enforcement of \eqref{g1const} and \eqref{g2const}
and the eigendecomposition are analogous to bi-5 case.
Corresponding to the unconstrained control points for the guides $\ggg$
(cf.\ \figref{fig:bi4bi3gd}),
we count 
\begin{align*}
\begin{tabular}{l l l}
degree of $\ggg$  & eigenfunctions \\
bi-5  & $17n+6+n^*$  \\
bi-4  & $9n+6+n^*$ \\
bi-3  & $3n+6+n^*$ 
\end{tabular}
\end{align*}

\def\wid{0.3\linewidth}
\def\widd{0.33\linewidth}
\begin{figure}[h]
   \centering
   \subfigure[bi4]{
   \includegraphics[width=\wid]{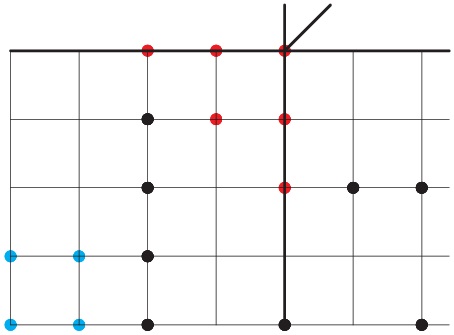}}
   \hskip1.0cm
   \subfigure[bi-3]{
   \includegraphics[width=\widd]{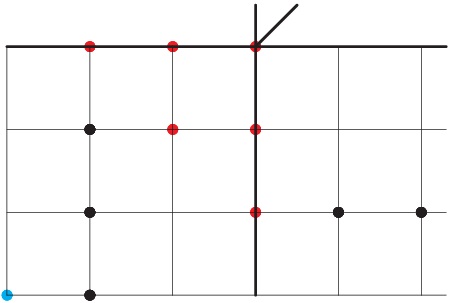}}
   \caption{
   The control points of $G^2$ guides of degree bi-4 and bi-3.
   Unconstrained points are marked as \textcolor{red}{red} and black disks.   
   }
   \label{fig:bi4bi3gd}
\end{figure}
\def\wid{0.3\linewidth}
\def\widd{0.22\linewidth}
\begin{figure}[h]
   \centering
   \subfigure[4 rings (guide bi-3)]{
   \includegraphics[width=\wid]{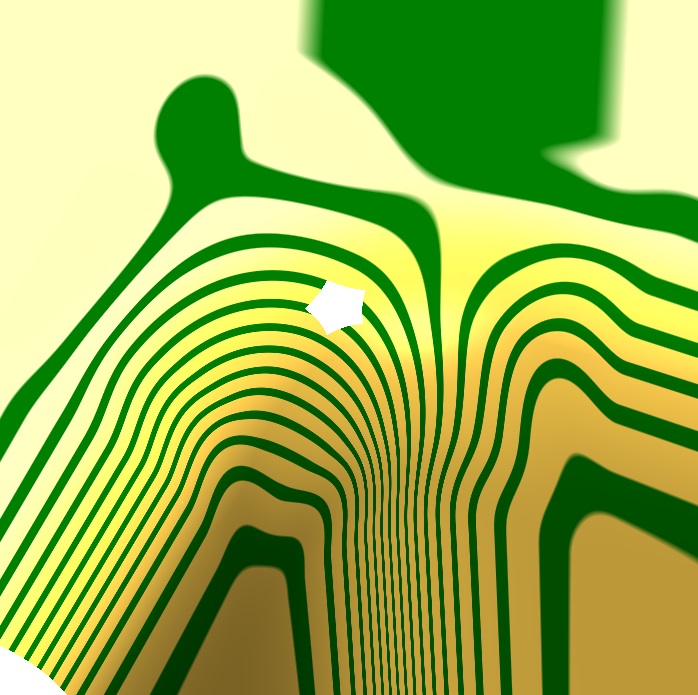}}
   \skp
   \subfigure[(bi-5) + 3 (bi-3)]{
   \includegraphics[width=\wid]{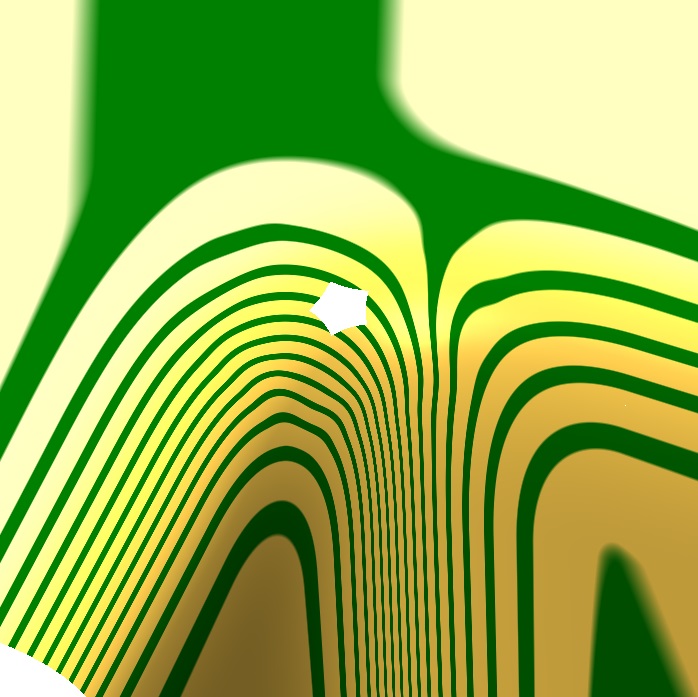}}
   \skp
   \subfigure[2 (bi-5) + 2 (bi-3)]{
   \includegraphics[width=\wid]{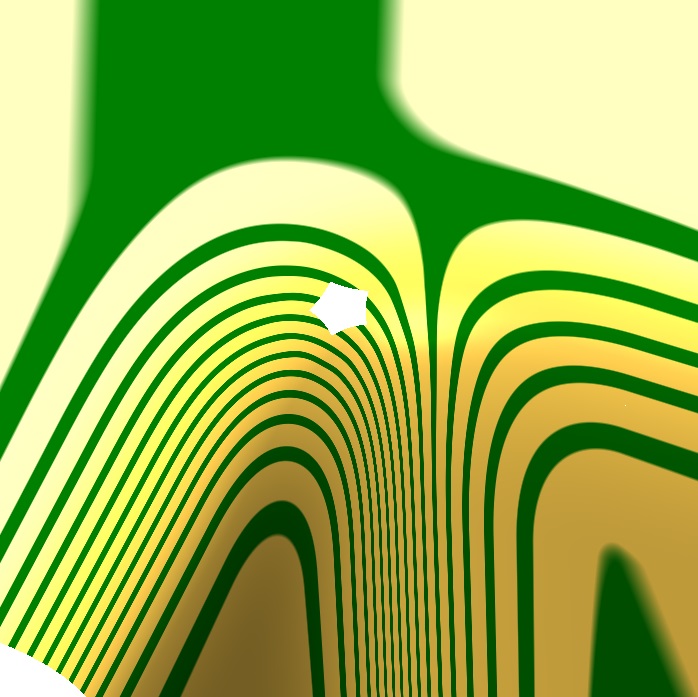}}
   \\
   \subfigure[rings]{
   \includegraphics[width=\widd]{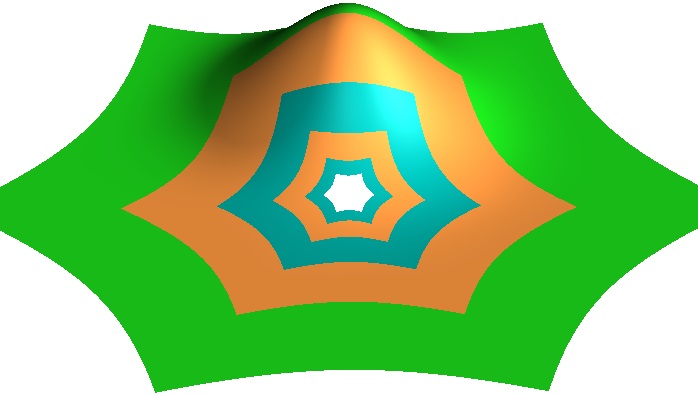}}
   \skp
   \subfigure[4 (bi-5)]{
   \includegraphics[width=\widd]{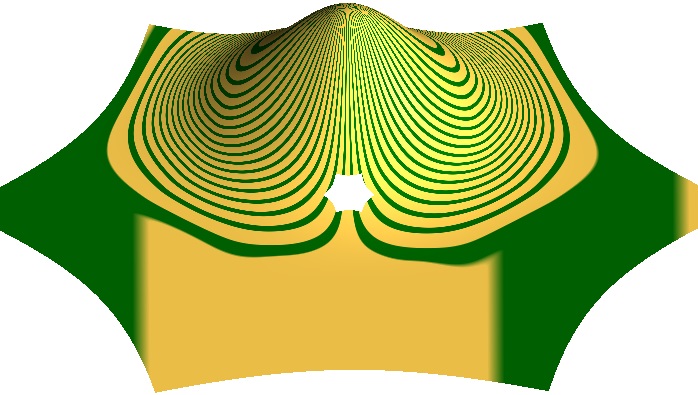}}
   \skp
   \subfigure[4 (bi-4)]{
   \includegraphics[width=\widd]{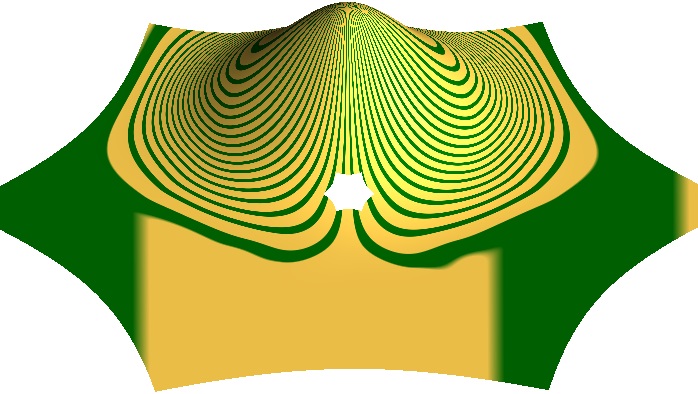}}
   \skp
   \subfigure[2(bi-5) +\quad 2 (bi-3)]{
   \includegraphics[width=\widd]{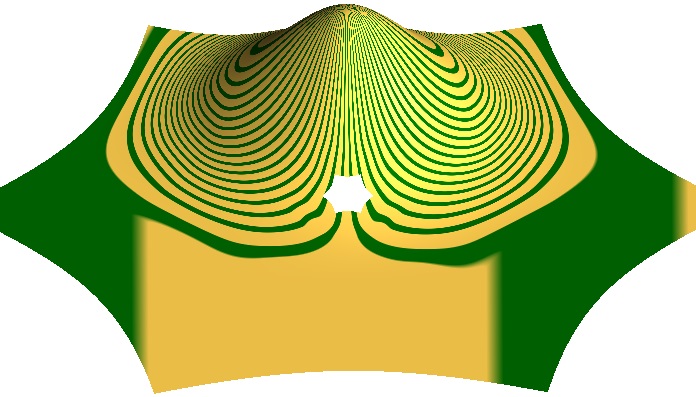}}
   \caption{
   {Comparison of bi-5 surfaces with 4 guided rings sampled from 
   guides of different degrees}. 
   The central gap is not filled with a cap.
   The \hld\ of guided subdivision surfaces from
   \IT\ row: net \figref{fig:ccmeshes}b, 
   \IB\ row: net \figref{fig:ccmeshes}h.   
   }
   \label{fig:bi543comp}
\end{figure}
On one hand, lower degree simplifies the analysis while on the 
other hand, lower degree of an $n$-patch guide
harms the surface quality as illustrated in
\figref{fig:bi543comp} (no caps are added since the poorer
\hld\ shows clearly in the subdivision rings).
\figref{fig:bi543comp}a shows that even an optimized bi-3 guide
yields oscillating and sharp \hl s.

We tested two approaches to improving the shape.
Applying one step of bi-5 subdivision followed by a bi-3 guide improves 
results, see \figref{fig:bi543comp}b; and twice repeating the 
bi-5 subdivision improves the \hld\ to \figref{fig:bi543comp}c. 
The \cnet\ of \figref{fig:ccmeshes}h, a single off-center spike, 
is notorious for revealing shortcomings of surface constructions.
Yet, the results (see \figref{fig:bi543comp} bottom row)
of applying a bi-4 guide (after zero or one step of bi-5 subdivision)
are difficult to distinguish from those of the bi-5 guide.
By contrast, (full enlargement of) \figref{fig:bi543comp}g
shows  oscillations of the \hl s for a bi-3 guide 
even after two bi-5 preprocessing steps.

\subsection{Bi-4 and bi-3 guided subdivision surfaces}\label{sec:bi4bi3sur}
For some applications, degree bi-5 may be considered a drawback. 
Leveraging the weak link between shape and smoothness,
we introduce $2\times 2$ macro patches for degree bi-4 and
$3\times 3$ macro patches for degree bi-3 (see \figref{fig:bi4bi3sch})
to enforce the $C^2$ prolongation between rings.
The resulting surfaces are curvature bounded at the central point
and preserve the \hld\ of the bi-5 construction well -- at the cost of 
increasing the number of polynomial pieces by a factor of 4, respectively 9.

Notably, for the corresponding hybrid construction,
the $G^1$ bi-3 caps are formally only $C^0$-connected
to the last guided ring. Yet, the resulting \hld\ is without flaw
(rightmost zoomed image of \figref{fig:bi34caps}b);
by contrast, enforcing 
exact $C^1$-continuity reduces the uniformity of the \hld.

\def\wid{0.3\linewidth}
\def\widd{0.15\linewidth}
\begin{figure}[h]
   \centering
   \subfigure[bi-4 ring and cap]{
   \includegraphics[width=\wid]{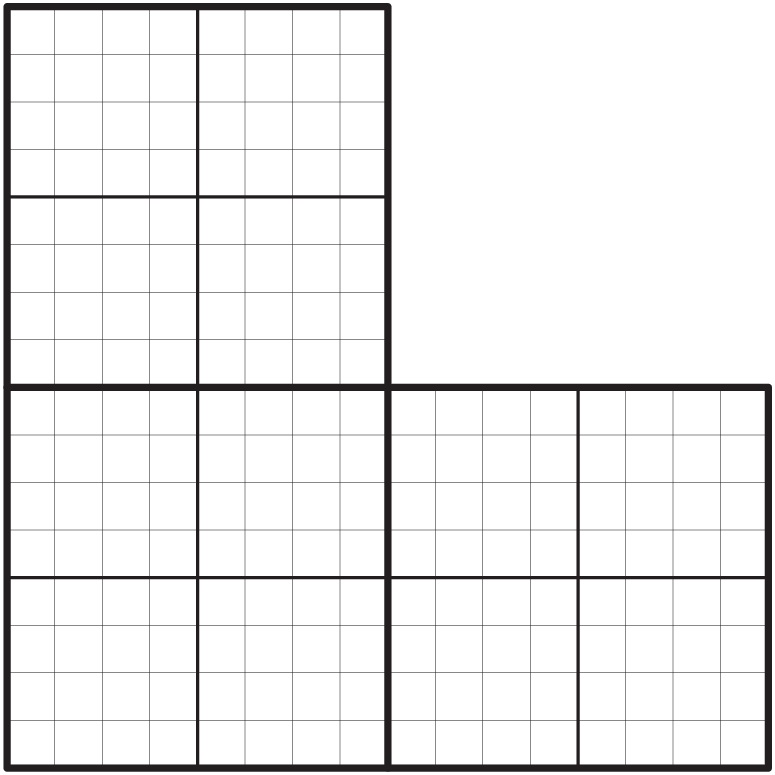}
   \begin{minipage}[b]{\widd}
   \includegraphics[width=\linewidth]{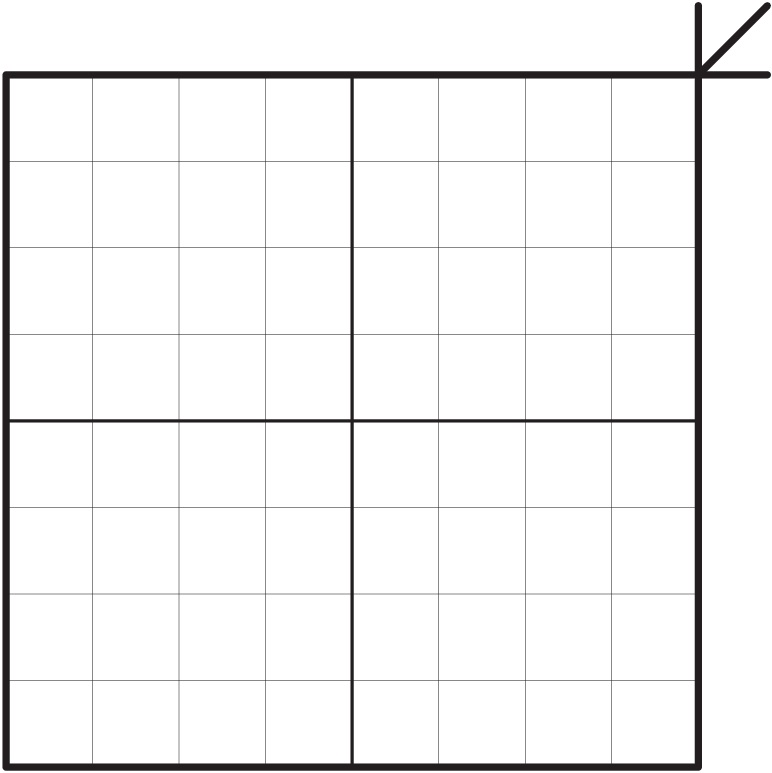}
   \vskip1cm
   \phantom{}
   \end{minipage}
   }
   \subfigure[bi-3 ring and cap]{
   \includegraphics[width=\wid]{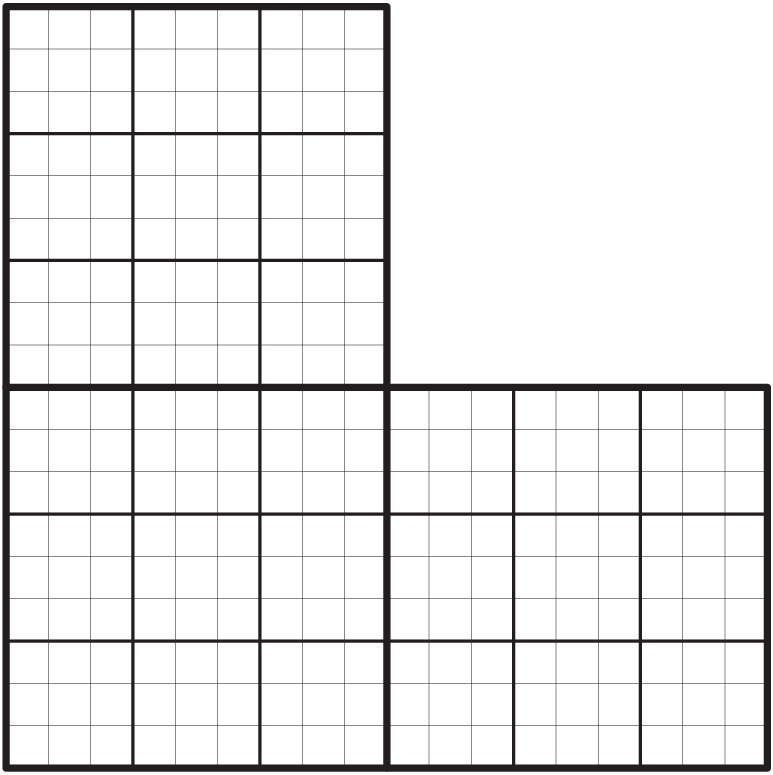}
   \begin{minipage}[b]{\widd}
   \includegraphics[width=\linewidth]{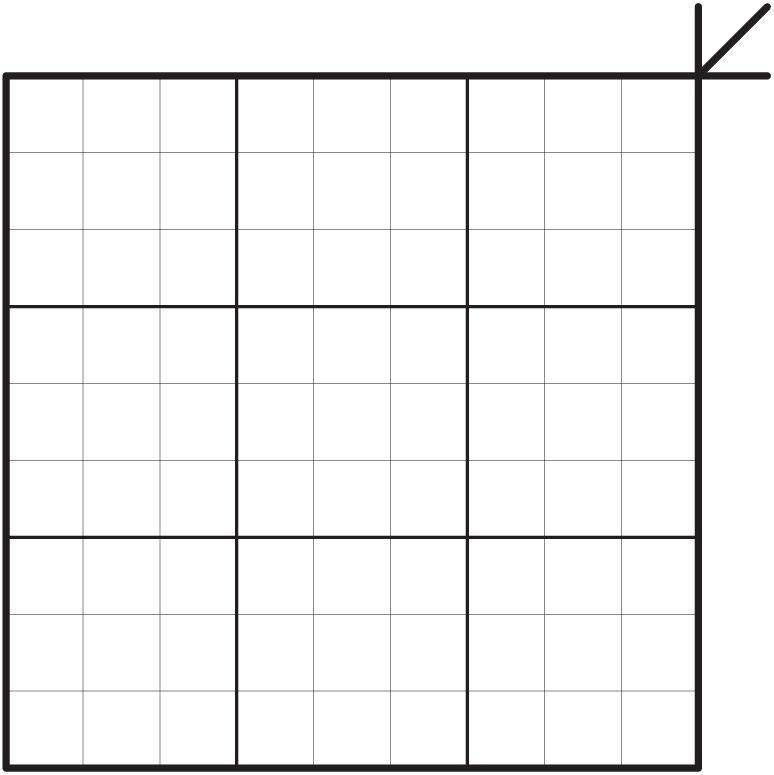}
   \vskip1cm
   \phantom{.}
   \end{minipage}
   }
   \caption{Macro-patches: one sector of the 
   last guided subdivision ring and a cap, consisting of $2\times 2$ 
   respectively $3\times 3$ pieces.
   }
   \label{fig:bi4bi3sch}
\end{figure}
\def\wid{0.2\linewidth}
\def\widd{0.1\linewidth}
\begin{figure}[h]
   \centering
   \subfigure[bi-3: \cnet, 4 guided rings, \hl s and Gauss curvature]{
   \includegraphics[width=\widd]{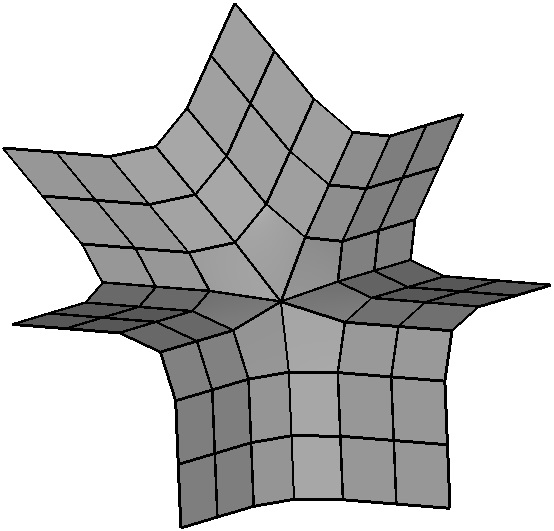}
   \includegraphics[width=\wid]{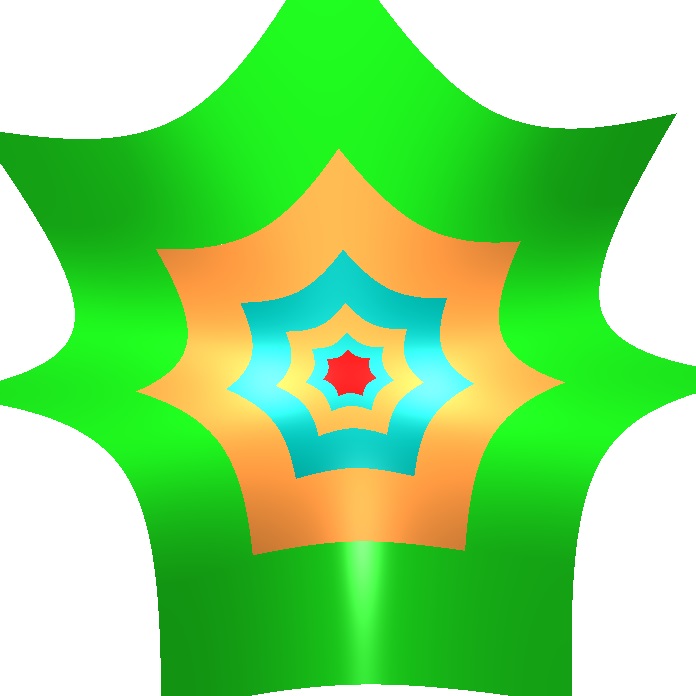}
   \includegraphics[width=\wid]{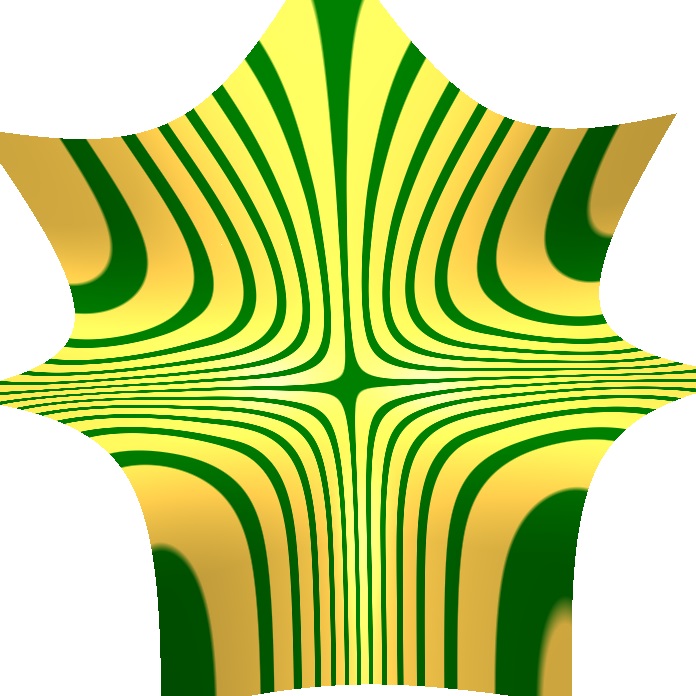}
   \includegraphics[width=\wid]{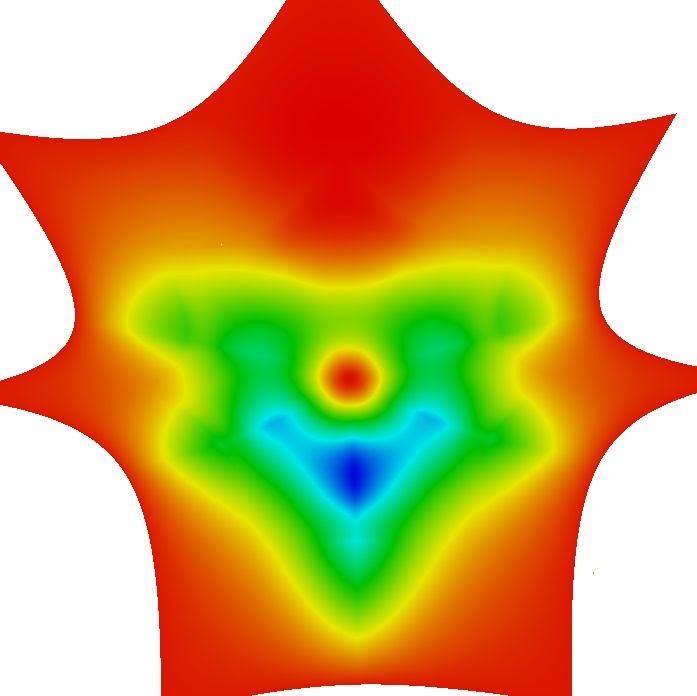}
   \includegraphics[width=\wid]{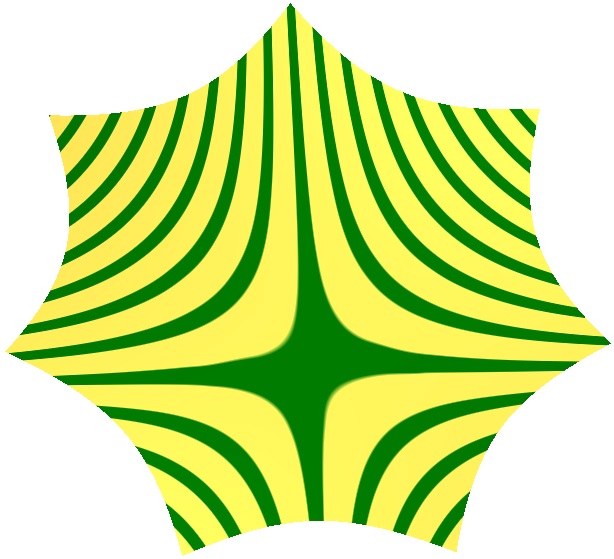}
   }
   \\
   \subfigure[bi-4: \cnet, 4 guided rings, \hl s and Gauss curvature]{
   \includegraphics[width=\widd]{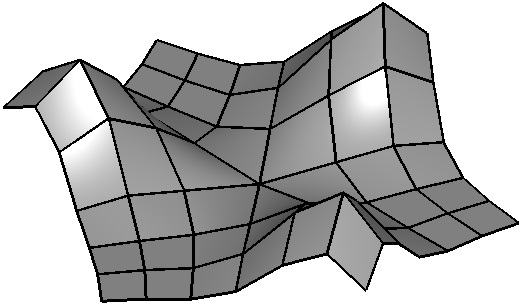}
   \includegraphics[width=\wid]{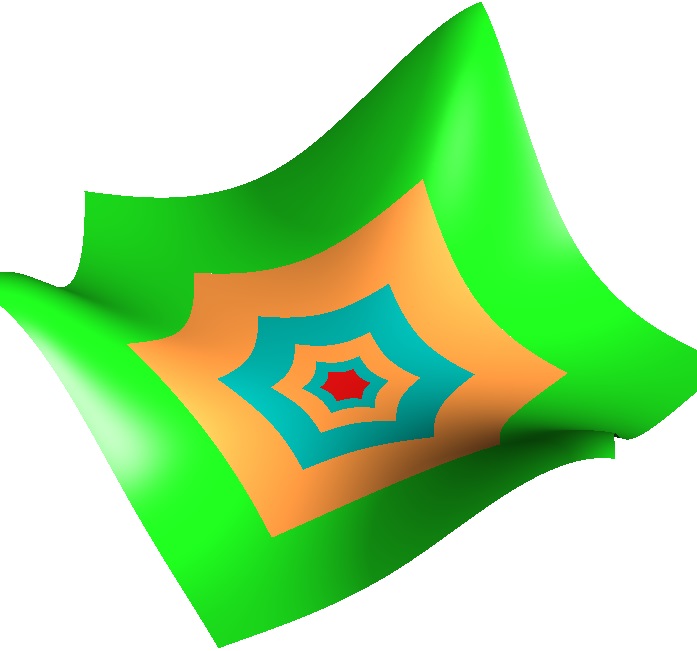}
   \includegraphics[width=\wid]{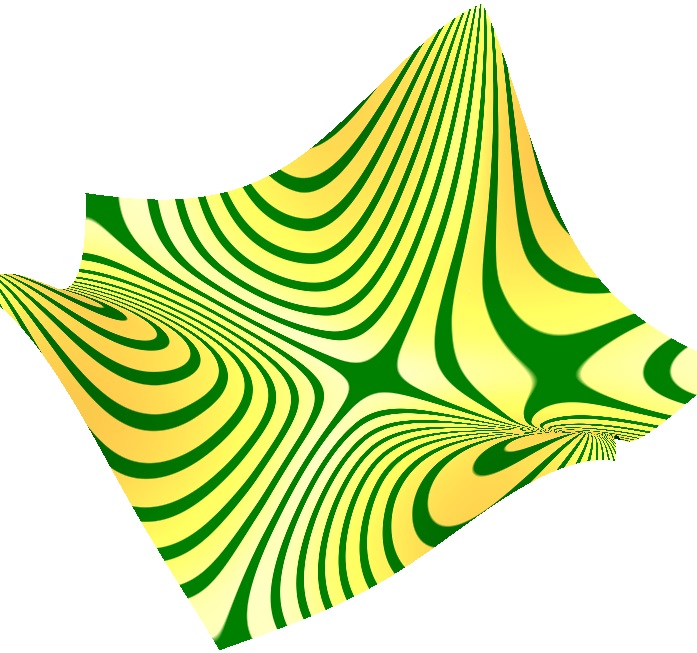}
   \includegraphics[width=\wid]{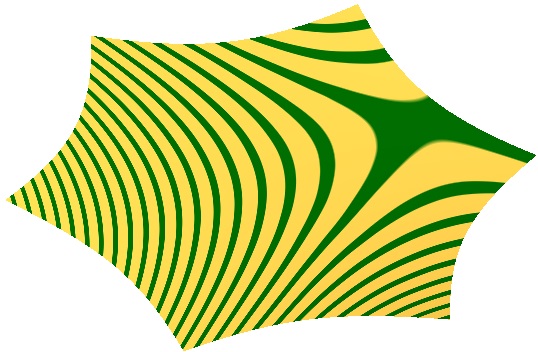}
   \includegraphics[width=\wid]{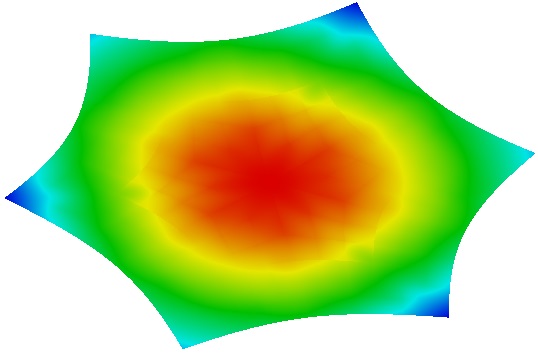}
   }
   \caption{Macro-patch constructions with (red) cap. 
   Rightmost figures are zoomed in to the last ring plus cap.
   }
   \label{fig:bi34caps}
\end{figure}


\subsection{Changing the contraction speed}
\label{sec:speedy}
Using the subdominant eigenvalue $\gl$ of \CCa\ subdivision
for $\gs$ implies that the contraction of guided rings becomes slower
when the valence increases ( \figref{fig:n893surf}b vs c).
Using instead the characteristic map of \cite{Karciauskas:2008:ASS},
the eigenvalue can be set to $\frac{1}{2}$ for all valencies $n>4$
to yield a uniform contraction speed.
\figref{fig:spchar} \IT\ vs \IB\ contrasts the characteristic map
\chg{for \CCa\ subdivision and $\gl$, with the uniform contraction
by $\frac{1}{2}$ of \cite{Karciauskas:2008:ASS}.}
%
\figref{fig:spn89}a shows the same surface as \figref{fig:n893surf}a
while \figref{fig:spn89}b is constructed 
with contraction speed $1/2$. The latter has visually identical
\hl s, an observation that holds for all \cnet s that we tested,
including all of \figref{fig:ccmeshes}
(The increased contraction is evident from the size of the 
\textcolor{red}{red} caps.)
\figref{fig:n893surf}b compares to 
\figref{fig:spn89}d.
All \cnet s of \figref{fig:ccmeshes} have a 
more uniform curvature distribution
in the vicinity of the caps when using speed $1/2$.
\def\wid{0.75\linewidth}
\begin{figure}[h]
   \centering
   \begin{overpic}[scale=.35,tics=10]{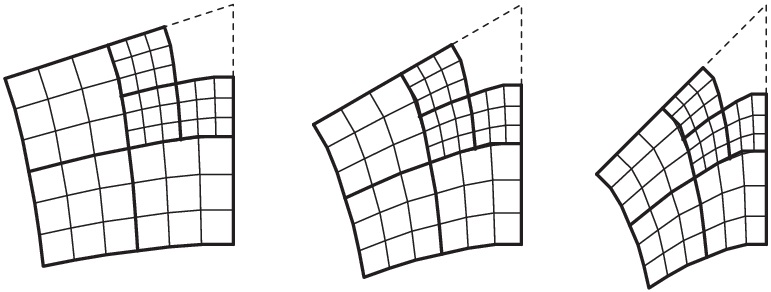}
        \put (15,1) {$n=5$}
	\put (55,1) {$n=6$}
	\put (90,1) {$n=8$}
	\end{overpic}\\
   \begin{overpic}[scale=.35,tics=10]{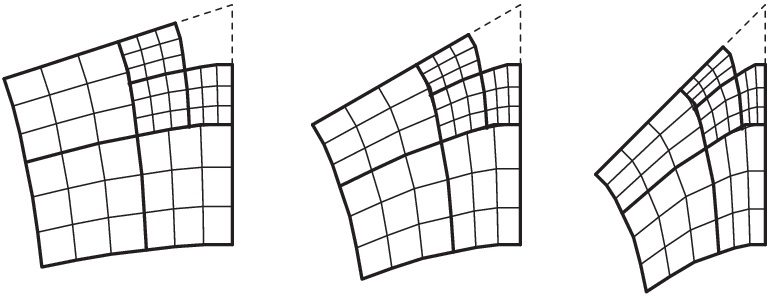}
        \end{overpic}
   \caption{
   {Characteristic maps.} \IT: \CCa-subdivision \cite{Catmull-1978-CC}; 
   \IB: adjustable speed subdivision \cite{Karciauskas:2008:ASS}
   with $\gl:=\frac{1}{2}$.   
   }
   \label{fig:spchar}
\end{figure}
\def\wid{0.22\linewidth}
\begin{figure}[h]
   \centering
   \subfigure[\cite{Catmull-1978-CC}]{
   \includegraphics[width=\wid]{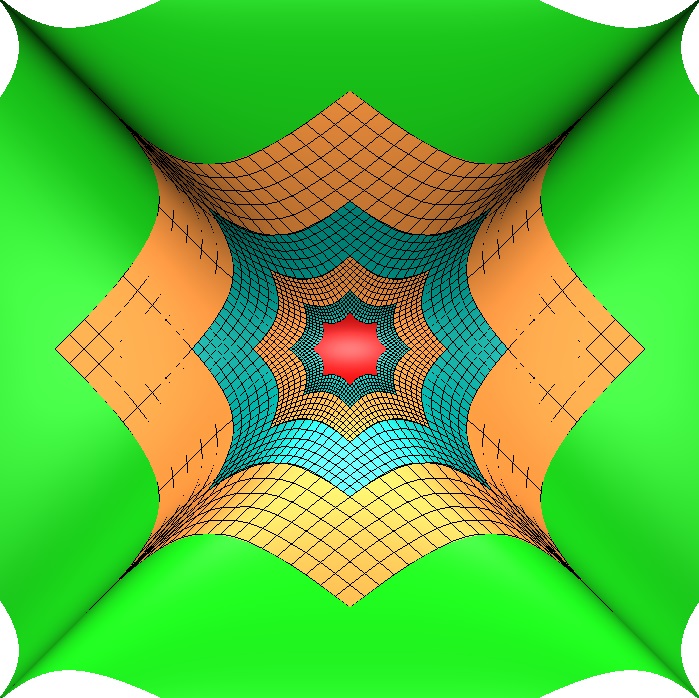}
   \includegraphics[width=\wid]{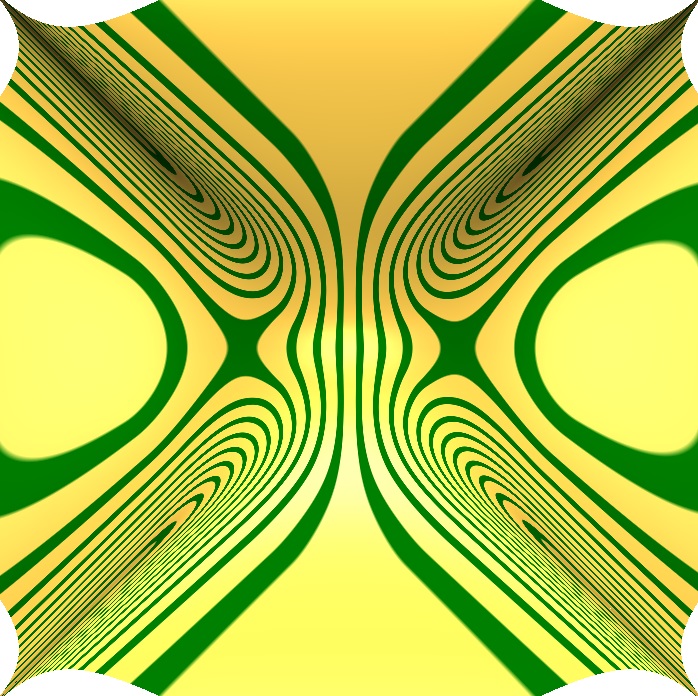}}
   \hskip0.0cm
   \subfigure[\cite{Karciauskas:2008:ASS}]{
   \includegraphics[width=\wid]{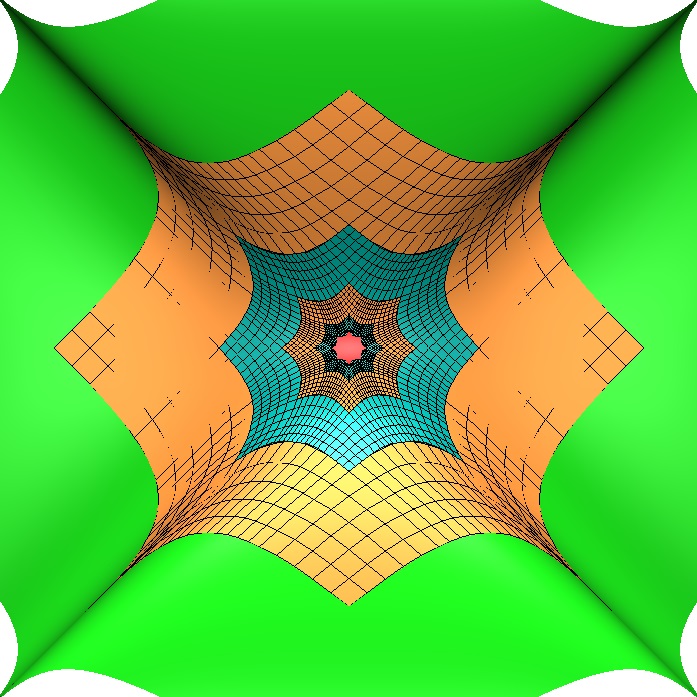}
   \hskip0.1cm
   \includegraphics[width=\wid]{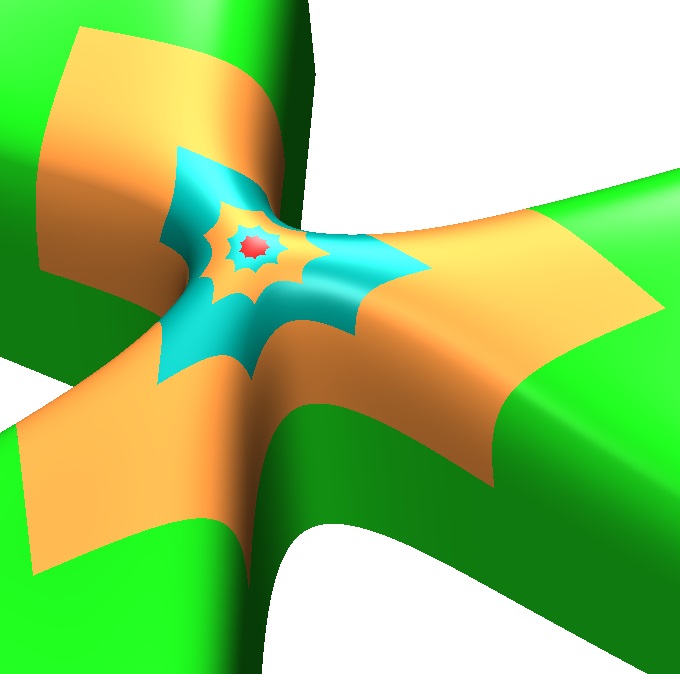}}
   \caption{
   \chg{
   Four guided rings are generated before adding the
   \textcolor{red}{red} cap. 
   The  \gss\ in
   (a) leverages the characteristic map of
   \cite{Catmull-1978-CC} while  (c) uses the more uniform
   contraction of \cite{Karciauskas:2008:ASS}.}
   (a,b \IL) net from \figref{fig:ccmeshes}e, 
   (b \IR) net from \figref{fig:ccmeshes}f.
   }
   \label{fig:spn89}
\end{figure}

\def\wid{0.4\linewidth}
\def\widd{0.45\linewidth}
\def\widdd{0.9\linewidth}
\begin{figure}
  \centering
   \subfigure[d.o.f.\ of bi-5 $C^2$ ring]{
   \begin{overpic}[scale=.23,tics=10]{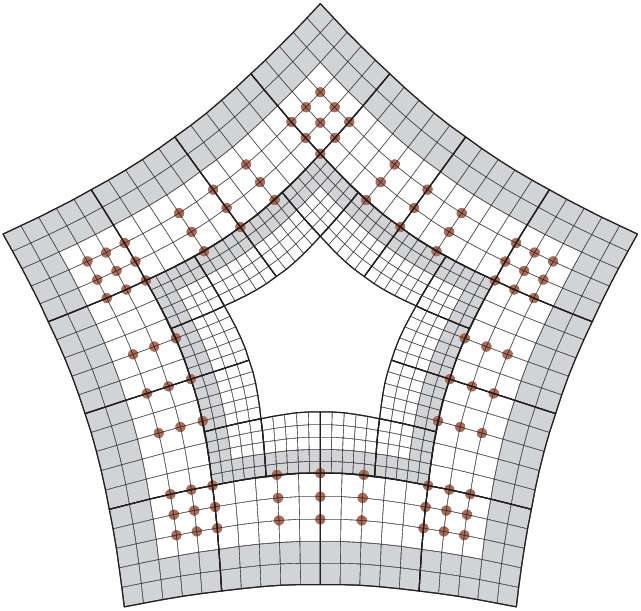}
        \put (87,53) {$1$}
	\put (71,57) {$2$}
	\put (63,58) {$3$}
   \end{overpic}
   }
   \subfigure[$\gfree$ and $\gfree$ of refined]{
   \includegraphics[width=\widd]{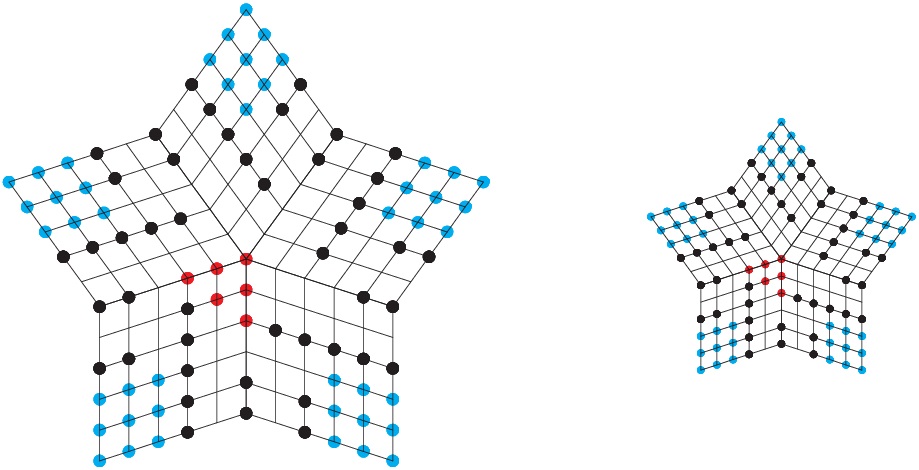}
   }
  \caption{Structure of refinable functions on guided subdivision
  surfaces. 
  (a) Degrees of freedom (brown bullets) in the $C^2$ bi-5 ring
  not influenced by the refined $\gfree$.
   The BB-coefficients in the gray layers are defined by 
   $C^2$ expansion from the surrounding surface.
  (b) A guide set $\gfree$ exist at each refinement step.
   }
   \label{fig:refine}
\end{figure}

\subsection{Refinement for functions on guided subdivisions surfaces}
\label{sec;refine}
With the shape of the subdivision surface determined
by the \cnet\ via $\gfree$ of the guide, 
here we define a nested space of refinable functions on the surface.
The combinatorial layout of the functions is identical to
that of the surfaces.
For example, each refinement of the bi-5 construction yields $18n$ new
degrees of freedom.
In \figref{fig:refine}a, the BB-coefficients determined by 
$C^2$ extension inwards are underlaid in gray and 
the $18n$ free coefficients of the
outermost subdivision ring that is no longer influenced 
by a once-refined set $\gfree$ are shown as brown bullets.
Functions corresponding to the markers 1,2,3
are displayed in \figref{fig:regfunc},b,c,d.
(Standard bi-5 respectively bi-6 spline subdivision with triple knots
can be applied to these rings in subsequent refinements.)
In addition to the $18n$ new degrees of freedom,
each subdivision step offers $\NN$ degrees of freedom corresponding to 
the set $\gfree$.
If the identity function is to be represented, the refined 
set $\gfree$ is obtained 
from its coarser predecessor \figref{fig:refine}b
via de Casteljau's algorithm.

While eigendecomposition can be used to obtain
finite expressions for computations on subdivision
surfaces, the considerable number of terms 
make us think that most numerical computations
are better served by computing with the hybrid 
representation
after a suitable number of refinement steps.
\def\widk{0.19\linewidth}
\def\wid{0.23\linewidth}
\begin{figure}[h]
   \centering
   \subfigure[filling]{
   \includegraphics[width=\widk]{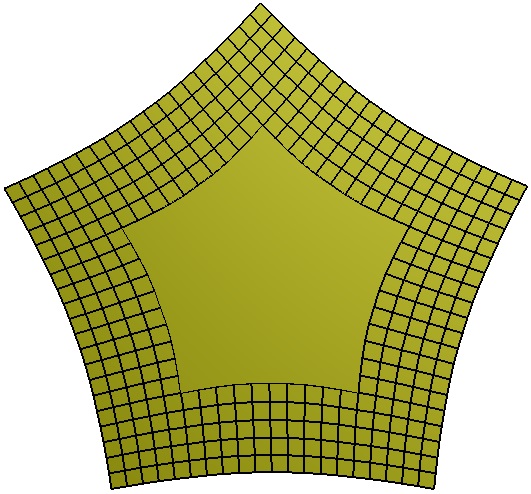}}
   \subfigure[$f_1$]{
   \includegraphics[width=\wid]{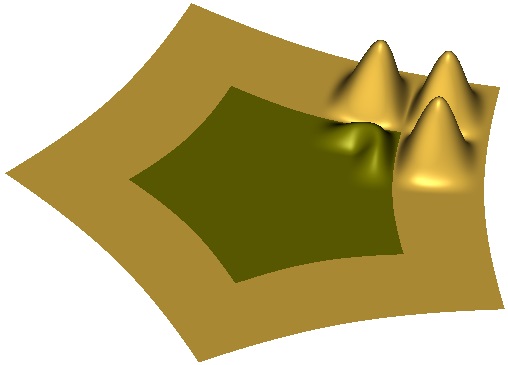}}
   \subfigure[$f_2$]{
   \includegraphics[width=\wid]{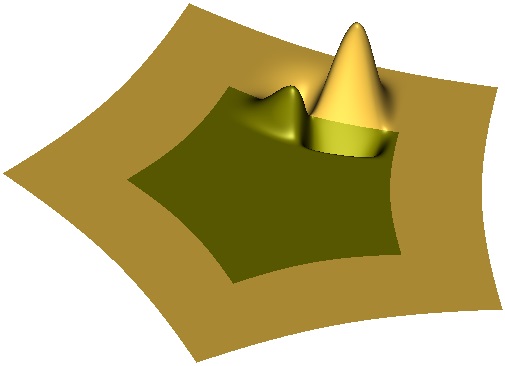}}
   \subfigure[$f_3$]{
   \includegraphics[width=\wid]{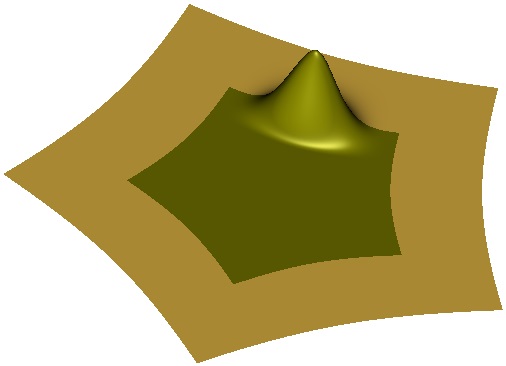}}
   \caption{
   (a) BB-coefficients of the innermost ring from \figref{fig:refine}a
   and the filling by subdivision.
   (b), (c), (d)
   Bi-5 functions $f_j$ with the meaning of the 
   subscripts indicated in \figref{fig:refine}a.
   }
   \label{fig:regfunc}
\end{figure}

\section{Conclusion}
The new \gss s offer an
automatic conversion of quad meshes with irregular vertices
into $C^2$ surfaces of good shape and built-in refinability.
A hybrid surface alternatively uses finitely many polynomial pieces
that preserve the shape
but are more readily amenable to subsequent computations on the surface.
The construction of \gss s is conceptually simple,
and has been implemented and tested on challenging examples. 
The eigen-structure of this class of subdivision algorithms 
is determined by the guide and fully analyzed.
The speed of contraction can be adjusted without harming the shape.

%
%

\bibliographystyle{eg-alpha-doi}
\bibliography{p}

\section*{A. Appendix}

Let
\[
\begin{aligned}
\dap_1:=& 4+8\ac+5\ac^2,\ \dap_2:=4+5\ac,\ \dap_3:=36+45\ac+5\ac^2,\\
\dap_4:=& 20+36\ac+9\ac^2-15\ac^3,\ \dap_5:=32+48\ac+15\ac^2,\\
\dap_6:=& 16+36\ac+36\ac^2+15\ac^3,\ \dap_7:=32+64\ac+72\ac^2+48\ac^3+15\ac^4.
\end{aligned}
\]
Then
\begin{align*}
\bpg_{30}:=&\frac{2}{3}\bpg_{10}+\frac{c-5}{3\ac}\bpg_{20}+\frac{5}{6\ac}\big(\bpg_{21}+\bpa_{21}\big);
\\
\bpa_{32}:=& \bpg_{32}-\frac{\dap_6}{3\dap_1}\big(\bpg_{11}-\bpa_{11}\big)
 +\frac{\dap_7}{3\ac\dap_1}\big(\bpg_{21}-\bpa_{21}\big)
 -\frac{\dap_6}{3\ac\dap_1}\big(\bpg_{22}-\bpa_{22}\big)\\
 & -\frac{\ac\dap_2}{2\dap_1}\big(\bpg_{42}-\bpa_{42}\big)
 +\frac{\ac^2}{2\dap_1}\big(\bpg_{52}-\bpa_{52}\big);
\\
\bpg_{31}:=& \frac{2\ac}{15}\bpg_{00}+\frac{8\ac}{15}\bpg_{10}+\frac{2}{3}\bpg_{11}
 -\frac{25+16\ac^2}{15\ac}\bpg_{20}\\
 & +\frac{1+2\ac}{6\ac}\bpg_{21}+\frac{3}{2\ac}\bpa_{21}
+\frac{1}{3\ac}\big(\bpg_{22}-\bpa_{22}\big)+\frac{2\ac}{5}\bpg_{40};
\\
\bpg_{41}:=& \frac{2\ac}{15\dap_1}\big(\ac\dap_2\bpg_{00}-\dap_3\bpg_{10}+5\dap_2\bpg_{11}\big)\\
& +\frac{4\dap_4}{15\dap_1}\bpg_{20}-\frac{\dap_5}{6\dap_1}\bpg_{21}
-\frac{\ac(16+25\ac)}{6\dap_1}\bpa_{21}+\frac{\dap_2}{3\dap_1}\big(\bpg_{22}-\bpa_{22}\big)\\
& +(1+\frac{3\ac}{5})\bpg_{40}+\frac{\dap_2}{4\dap_1}\big(\bpg_{42}-\bpa_{42}\big)
  +\frac{\ac}{5}\bpg_{50}+\frac{\ac}{4\dap_1}\big(\bpg_{52}-\bpa_{52}\big);
\\
\bpg_{51}:=& \frac{2\ac^2}{3\dap_1}\big(\ac\bpg_{00}+(3\ac-5)\bpg_{10}+5\bpg_{11}\big)\\
& +\frac{\ac}{6\dap_1}\big(-16\ac^2\bpg_{20}+5(\ac-4)(\bpg_{21}-\bpa_{21})+10(\bpg_{22}-\bpa_{22})\big)\\
& -\ac\bpg_{40}+\frac{5\ac}{4\dap_1}\big(\bpg_{42}-\bpa_{42}\big)
  +(1+\ac)\bpg_{50}+\frac{4+3\ac}{4\dap_1}\big(\bpg_{52}-\bpa_{52}\big).
\end{align*}
For $i=3,4,5$, we get $\bpa_{i1}$ from $\bpg_{i1}$ by swapping
$\bpg$ with $\bpa$.

\section*{B. Appendix: hybrid caps of degree bi-6}
The bi-6 caps are internally $G^1$ according to
\begin{equation}
\partial_v\bfa+\partial_v\bfg-2\ac(1-u)^2\partial_u\bfg=0.
\label{c1g1const}
\end{equation}
and they are $C^1$-connected to the last guided bi-6 surface ring. 
With the notations and indexing of \figref{fig:bi6tanen}a in the 
guide construction of \secref{sec:guide}, the unconstrained
coefficient of the local solution to \eqref{c1g1const} 
are marked as bullets in \figref{fig:bi6tanen}a.
\begin{figure}[h]
   \centering
   \subfigure[indexing]{
   \begin{overpic}[scale=.2,tics=10]{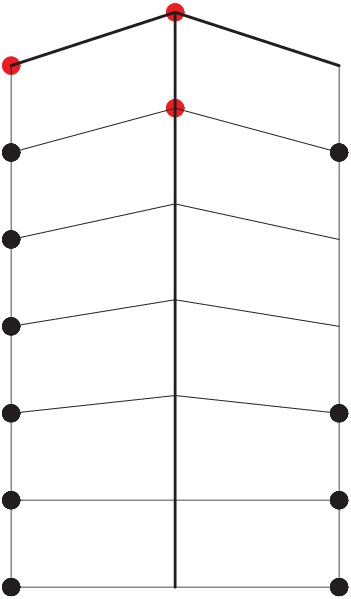}
        \put (25,100) {$00$}
	\put (-8,90) {$01$}
	\put (58,90) {$01$}
	\put (-8,75) {$11$}
	\put (58,75) {$11$}
	\put (-8,60) {$21$}
	\put (58,60) {$21$}
	\put (-8,0) {$61$}
	\put (59,0) {$61$}
	\put (24,82) {$10$}
	\put (24,67) {$20$}
	\put (24,51) {$30$}
	\put (24,1) {$60$}
	\put (-10,40) {$\bpg$}
	\put (62,40) {$\bpa$}
	\end{overpic}}
   \hskip1.5cm	
   \subfigure[reparameterization $\hat\gs$]{
   \begin{overpic}[scale=.3,tics=10]{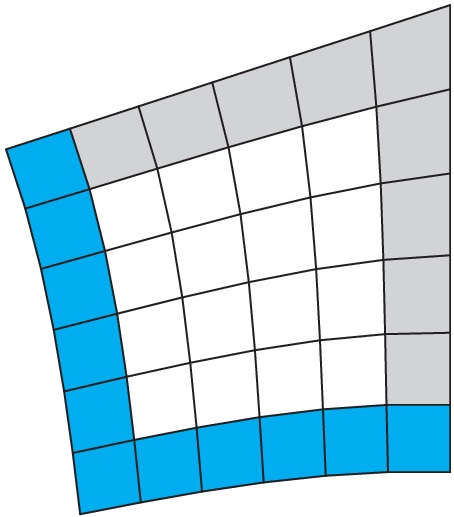}
        \end{overpic}}	
   \caption{
   (a) Local symmetric indexing: the unconstrained BB-coefficients are 
   marked as \textcolor{red}{red} and black disks.
   (b) bi-6 reparameterization $\hat\gs$ for sampling the guide.   
   }
   \label{fig:bi6tanen}
\end{figure}
The interactions between the $n$ local $G^1$ systems of equations
at the irregular point $\bpa_{00}:=\bpg_{00}$ are resolved
by selecting three BB-coefficients in one sector
(\textcolor{red}{red} disks in \figref{fig:bi6tanen}a)
to define the tangent plane at the irregular point
and define the corresponding BB-coefficients in
other sectors recursively as 
\begin{equation}
   \bpa_{10}:=\bpg_{10}\ ,\quad 
   \bpa_{01}:= -\bpg_{01}+2\ac \bpg_{10}+2(1-\ac)\bpg_{00}\ .
\label{cenlin}
\end{equation}
The explicit formulas for the dependent points of the local solution
are 
\begin{equation}
\label{bi6solv}
\begin{aligned}
\bpg_{20}:=& \frac{3}{5\ac}\big(\bpg_{11}+\bpa_{11}\big)+
\big(\frac{6}{5}-\frac{6}{5\ac}\big)\bpg_{10}-\frac{1}{5}\bpg_{00}\ ;\\
\bpa_{30}:=& \frac{1}{20}\big(\bpg_{00}+\bpg_{60}\big)-\frac{3}{10}\big(\bpg_{10}+\bpg_{50}\big)
+\frac{3}{4}\big(\bpg_{20}+\bpg_{40}\big)\ ;\\
\bpg_{40}:=& \frac{1}{2}\big(\bpg_{41}+\bpa_{41}\big)+\frac{\ac}{15}\big(\bpg_{50}-\bpg_{60}\big)\ ;\\
\bpg_{50}:=& \frac{1}{2}\big(\bpg_{51}+\bpa_{51}\big)\ ,\quad \bpg_{60}:=\frac{1}{2}\big(\bpg_{61}+\bpa_{61}\big)\ ;\\
\bpa_{21}:=& -\bpg_{21}-\frac{\ac}{15}\bpg_{00}+\frac{2\ac}{5}\bpg_{10}+(2-\ac)\bpg_{20}+\ac\bpg_{40}
-\frac{2\ac}{5}\bpg_{50}+\frac{\ac}{15}\bpg_{60}\ ;\\
\bpa_{31}:=& -\bpg_{31}+\frac{1}{10}\bpg_{00}-\frac{3}{5}\bpg_{10}+\frac{3}{2}\bpg_{20}+\frac{3-\ac}{2}\bpg_{40}
+\frac{1-\ac}{10}\big(\bpg_{60}-6\bpg_{50}\big)\ .
\end{aligned}
\end{equation}
A bi-6 reparameterization $\hat\gs$ for sampling the guide is 
rotationally and sector bisectrix symmetric 
and the outer BB-coefficients (blue underlay in \figref{fig:bi6tanen}b)
$C^1$-extend the characteristic ring of \CCa-subdivision.
This leaves $14$ free parameters that are set to minimize the 
sum up to fifth derivatives,
$
\int^1_0 \int^1_0
   \sum_{i+j=5, i,j\ge 0}
      \frac{5!}{i!j!} (\partial^i_s \partial^j_t f(s,t))^2 ds dt
$.
Applying De Casteljau at $u=\gl^\ell=v$ to the sector of the guide
and sampling 
$\pc{\ggg^k\gl^r\circ{\hat\gs}}{6}{4\times 4}{}$ 
at all four corners of $\hat\gs$
to form the bi-6 patch according to \figref{fig:herm56col}b
implies that the resulting cap $\sixcap^k$
inherits the unique quadratic expansion of the guide.
$C^1$-extending the last guided ring leaves
$\bpg_{21}$, $\bpg_{31}$, $\bpa_{21}$, $\bpa_{31}$
(see \figref{fig:bi6tanen}a)
to be the least squares best fit to 
the corresponding BB-coefficients of $\sixcap^k$ and 
$\sixcap^{k+1}$.
As for $\ggg$ pretabulation simplifies practical computation.

Although the construction is formally only $G^1$,
it is curvature continuous at the center and this 
partly accounts for its good shape.

\section*{C. Appendix: Eigenanalysis of the subdivision algorithm}
Since the central point stays fix, the dominant eigenvalue is $1$.
\figref{fig:eigenstr} lists the indices of the other
unconstrained control points $\gfree$ of the guide $\ggg$ 
(recall that \textcolor{red}{red} bullets labeled $1,\ldots,5$ are only
unconstrained for sector $k=0$ and that for $n=3,6$ there is an additional
degree of freedom at the location marked by a circled cross
in \figref{fig:gdtanen}b). 
With the abbreviations
\begin{align*}
   n^* &:= \begin{cases} 
   n+1, & \text{ for } n \in \{3,6\},\\
   n, & \text{ else },
   \end{cases},
   \quad \mmm := 6+n^*, 
   \\
   \kbar &:=6+k;\quad \khat:=\mmm+k;\quad 
   dn^+:=dn+\khat,\ d=1,\ldots,16,
\end{align*}
we label the $\NN := 17n+\mmm-1$ elements of $\gfree$
as illustrated in  \figref{fig:eigenstr}.
\begin{figure}[h]
   \centering
   \begin{overpic}[scale=.6,tics=10]{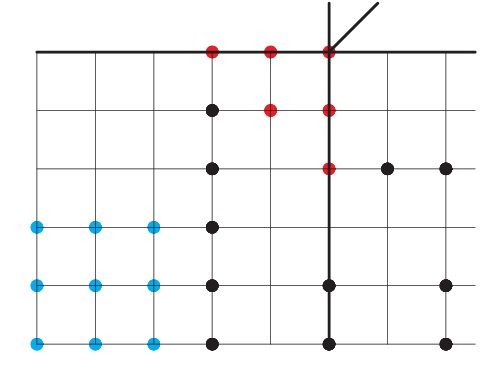}
        \put (53,66) {$1$}
	\put (41,66) {$3$}
	\put (62,53) {$2$}
	\put (50,53) {$4$}
	\put (62,41) {$5$}
	\put (38,53) {$\kbar$}
	\put (38,41) {$\khat$}
	\put (70,43) {$\kbar+1$}
	\put (92,41) {$\khat+1$}
	\put (60,19) {$\nnp$}
	\put (61,0) {$2\nnp$}
	\put (83,19) {$5\nnp$}
	\put (83,0) {$8\nnp$}
	\put (43,31) {$3\nnp$}
	\put (43,19) {$4\nnp$}
	\put (43,0) {$7\nnp$}
	\put (27,31) {$6\nnp$}
	\put (25,19) {$10\nnp$}
	\put (25,0) {$13\nnp$}
	\put (16,31) {$9\nnp$}
	\put (14,19) {$12\nnp$}
	\put (12,0) {$15\nnp$}
	\put (-1,31) {$11\nnp$}
	\put (-1,19) {$14\nnp$}
	\put (-1,0) {$16\nnp$}
	\put (2,42) {$\bp^k$}
	\put (93,29) {$\bp^{k+1}$}
	\end{overpic}
   \caption{
   Indices for the eigen-analysis.
   }
   \label{fig:eigenstr}
\end{figure}

After application of de Casteljau at $\gl$,
the linearly-reparameterized bi-5 patches satisfy
the \emph{unchanged} constraints (\ref{g1const}) and (\ref{g2const});
this was intended and is verified by inspection. 
The mapping of $\gfree$ to its next-level counterpart
yields systems of linear eigen-equations $\xeq^s_i$ 
in unknowns $x^s_j$, $i,j=1,\ldots,\NN$,
that form the eigenvectors 
corresponding to eigenvalues $\gl^s$, $s=1,\ldots,10$.
Solving the large and highly underconstrained
systems with   
symbolic entries $\gl$ defies the capabilities of Maple,
hence requires some careful guesses
based on an observed pattern for concrete $n$ and 
$\gl$. The underconstrained systems are reduced
by judiciously setting various $x^s_i$ to zero and solving,
for the specific $\gl:=\frac{1}{2}$, a subset of 
(system, variables)-pairs ($\xeq^s_i$, $x^s_j$).
We abbreviate 
\[
a:b\quad := a, a+1,\ldots,b,
\qquad
\KK(\ga,\gb) := \chg{(\mmm + \ga n: \mmm+\gb n -1)}.
\]
For $s=1:10$, 
we list parameters set to zero, pairs of equations and variables, and the 
free parameters that will characterize the eigenvectors:
\begin{align}
\begin{matrix}
s= & x^s_{1:j}=0 & (\xeq^s_{i:\NN},x^s_{i:\NN}) 
   & x^s_k \text{ free}\\
   & j= & i= & k= \\
1 &  & 3 & 1,2\\
2 & 2 & 6 & 3:5\\
3 & 5 & \mmm &  6:m-1\\
4 & \mmm-1 & 2n+\mmm & \KK(0,2)\\
5 & 2n+\mmm-1 &  (I_5,J_5) & \KK(2,4)\cup \KK(5,6)\\
6 & 4n+\mmm-1& 7n+\mmm & \KK(4,7) \\
7 & 7n+\mmm-1& (I_7, I_7) & \KK(7,8)\cup \KK(9,11)\\
8 & 11n+\mmm-1& 14n+\mmm & \KK(11,14)\\
9 & 14n+\mmm-1& 16n+\mmm & \KK(14,16)\\
10 & 16n+\mmm-1 &  & \KK(16,17)
\end{matrix}
\label{eq:free}
\end{align}
For $s=5$, system indices and variables differ:
$ I_5 := \KK(5,17)$ and the variables are the union of labels
$ J_5 := \KK(4,5) \cup\ \KK(6,17)$.
For $s=7$,  $I_7 := \KK(8,9)\cup\  \KK(11,17)$.
The system of equations $\xeq^{10}_i$ is not listed since it
is satisfied by setting $x^{10}_{1:j}=0$.
The solutions of these systems 
are substituted into the initial systems with 
symbolic $\gl$ to verify that they solve the equations for
any choice of $\gl$. 
This yields explicit formulas for the eigenvectors
in terms of the free parameters listed in the
right column of \eqref{eq:free}.
The eigenvectors, one per free parameter, span the eigenspace 
with the eigenvalues listed in Table \eqref{eq:eigvals}.


\end{document}